\documentclass{article}

\fussy

\usepackage{dsfont}
\usepackage{graphicx}
\usepackage{amsmath}
\usepackage{amsfonts}
\usepackage{amssymb}
\usepackage{amsthm}
\usepackage{xcolor}
\usepackage{bbm}
\usepackage{algorithm}
\usepackage{algpseudocode}
\usepackage{cite}
\usepackage{hyperref}

\title{Near Optimal Finite Model POMDP Approximation of Controlled Diffusion with Partial Information}

\author{
  Yunus Emre Demirci\thanks{Department of Mathematics, University of Michigan, Ann Arbor, MI, USA. \texttt{edemirci@umich.edu}}
  \and
  Somnath Pradhan\thanks{Department of Mathematics, Indian Institute of Science Education and Research (IISER) Bhopal, Bhopal, Madhya Pradesh, India. \texttt{somnath@iiserb.ac.in}}
  \and
  Serdar Y\"{u}ksel\thanks{Department of Mathematics and Statistics, Queen's University, Kingston, ON K7L 3N6, Canada. \texttt{yuksel@queensu.ca}}
}

\usepackage{graphicx} % Required for inserting images
\usepackage{amsmath}
\usepackage{amsfonts}
\usepackage{comment}
\usepackage{xcolor}
\usepackage{graphicx}%
\usepackage{multirow}%
\usepackage{amsmath,amssymb,amsfonts}%
\usepackage{amsthm}%
\usepackage{mathrsfs}%
\usepackage[title]{appendix}%
\usepackage{xcolor}%
\usepackage{textcomp}%
\usepackage{manyfoot}%
\usepackage{booktabs}%
\usepackage{algorithm}%
\usepackage{algorithmicx}%
\usepackage{algpseudocode}%
\usepackage{listings}%
\usepackage{amsmath,nicefrac,upgreek,mathtools,url}
\usepackage{amsthm,amsmath,amsfonts,amssymb}
\usepackage[numbers]{natbib}
\usepackage{verbatim}

\newcommand{\bea}{\begin{eqnarray}}
\newcommand{\ena}{\end{eqnarray}}
\newcommand{\beas}{\begin{eqnarray*}}
\newcommand{\enas}{\end{eqnarray*}}
\newcommand{\beq}{\begin{equation}}
\newcommand{\enq}{\end{equation}}

\newcommand{\X}{\mathbb{X}}
\newcommand{\Y}{\mathbb{Y}}
\newcommand{\U}{\mathbb{U}}

\newcommand{\E}{\mathbb{E}}
\newcommand{\Pref}{\mathbb{P}_{ref}}

\newcommand{\cP}{\mathcal{P}}

\newcommand{\R}{\mathbb{R}}
\newcommand{\cT}{\mathcal{T}}
\newcommand{\cK}{\mathcal{K}}

\usepackage{bbm}

\usepackage{color}

\newtheorem{theorem}{Theorem}%  meant for continuous numbers
\newtheorem{remark}{Remark}%

\newtheorem{assumption}{Assumption}%
\newtheorem{corollary}{Corollary}%
\newtheorem{lemma}{Lemma}%

\begin{document}

\maketitle
\begin{abstract}
We study continuous-time stochastic control problems under partial information. Our main contributions are the following: (i) We show near optimality, under finite horizon as well as discounted infinite horizon cost criteria, of piece-wise constant control policies obtained via solutions to discrete-time partially observable Markov decision process (POMDP) approximations, constructed via Euler-Maruyama approximations of both state and observation processes; (ii) we obtain several regularity properties of such POMDP approximate models; and (iii) using these regularity properties, we  arrive at near optimal finite model approximations. (iv) These culminate into our main result which is that an optimal solution for the approximate finite POMDP model is near optimal for the original diffusion with partial information. 
%First we give a discrete-time approximation introduced in \cite{?} and an Euler--Maruyama approximation for both the state and observation processes and a belief-quantized approximation. 
%In particular, we show that an optimal policy for the finite POMDP is near optimal for the original diffusion. 
\end{abstract}
    
\section{Introduction}

Controlled diffusion processes with partial observation have wide-ranging applications in finance, engineering, and operations research. In these settings, a decision-maker seeks to minimize a cost criterion by selecting actions based on noisy, partial observations of an underlying continuous-time state process. The primary difficulty in these problems arises from the information structure: unlike fully observed models, the decision-maker must rely on the observation history or the belief state (the conditional distribution of the state).

While the theoretical foundations of such problems are well established (see \cite{kushner2014partial} for an excellent survey) using wide-sense admissible policies and measure transformation techniques \cite{benevs1971existence,davis1972information,davis1973dynamic, FlPa82}, obtaining explicit optimal or near optimal policies remains a challenge. Specifically, Fleming and Pardoux \cite{FlPa82} obtain existence via a belief-separation approach, whereas Bismut \cite{bismut1982partially} derives existence without relying on separation. Although not stated as such, both frameworks enforce admissibility by constructing a suitable defined Young topology on the joint space of measurement and control paths. Within this topology, the measurement process is constrained to a fixed Brownian marginal, and a conditional independence condition is imposed: future observation increments are independent of past controls and measurements to preserve the non-anticipativity of the policy. Different policy classes (on what filtration actions can be measurable on) impact conditions on the existence of solutions for controlled diffusions under partial information and whether measure change arguments can be applicable \cite{cohen2025optimal}. 

Because exact solutions are intractable, continuous-time models are typically approximated by discrete-time counterparts in practice. However, the existing discrete-time approximation literature mostly focuses on fully observed systems. Comprehensive analyses and error bounds for fully observed diffusions are provided by Kushner and Dupuis  \cite{kushner1990numerical}, \cite{kushner2001numerical}, \cite{kushner2012weak} and Krylov \cite{krylov1999approximating,krylov2000rate, krylov2008controlled} alongside other contributions \cite{Jakobsen2006, jakobsen2019improved, PY25discapprox}. 

In the uncontrolled setting, the discrete-time approximation of partially observed models (the nonlinear filtering problem) has been studied: Di Masi and Runggaldier \cite{diMasi81, diMasi82} and Kushner \cite{kushner2012weak} proved the weak convergence of the discrete-time filter to its continuous-time limit. Similar updating procedures have been deployed for financial allocation problems \cite{martinez2009}. Crisan et al. \cite{crisan2022} recently developed high-order pathwise approximations that establish the robustness of the discretized solutions building on continuous dependence of filters in (path-valued) measurements under the supremum norm \cite{bain2009fundamentals}. For further numerical methods on continuous-time stochastic control with partial information, also involving discrete-time approximations as well as HJB inspired computational methods, see \cite{tallec2019making,alt2020pomdps}. We also note that there is an intrinsic relation between approximability and filter stability, and robustness to model perturbations \cite{chigansky2009intrinsic,cohen2023exponential}.

Despite these advances, a rigorous justification for the discrete approximations for controlled  diffusions with partial information still entails several open problems. One must rigorously ensure that a policy designed for the approximate discrete model remains near-optimal when applied to the original continuous-time system.

\subsection*{Contributions}
In this paper, we provide a comprehensive approximation framework for partially observable controlled diffusions, establishing a rigorous chain of approximations that leads to implementable, finite-state algorithms. We construct a discrete-time Euler-Maruyama approximation for the partially observed model and prove that the optimal policy derived from this discretized formulation is near-optimal for the original continuous-time diffusion. Our main contributions are as follows:

\begin{enumerate}
    \item \textbf{Near Optimality of Euler-Maruyama Discrete Time Approximations.} 
While the precise mathematical details are presented later in the paper, as one of our main results, for a controlled diffusion model with partial information, given by the following dynamics 
\begin{equation}\label{cont_model0}
dX_t = b(X_t, U_t) \, dt + \sigma(X_t) \, dW_t, \quad X_0 = x \in \mathbb{R}^d,
\end{equation}
with observations
\begin{align}\label{obs_model0}
dY_t = g(X_t) \, dt + dB_t,
\end{align}
we show that policies obtained for the discrete-time model (under a corresponding approximate discrete-time criterion)
\begin{align}
X^{EM}_{kh} &= X^{EM}_{kh-h} + b(X^{EM}_{kh-h}, U_{kh-h})h + \sigma(X^{EM}_{kh-h}) (W_{kh}-W_{kh-h}) \, \label{eq:euler_state0} \\
Y^{EM}_{kh} &= Y^{EM}_{kh-h} + g(X^{EM}_{kh-h})h + (B_{kh}-B_{kh-h}),\label{eq:euler_obs0}
\end{align}
and applied to (\ref{cont_model0}-\ref{obs_model0}) for sufficiently small $h$, are near optimal under finite horizon as well as discounted infinite horizon cost criteria. 

To this end, revisiting \cite{pradhanyuksel2023DTApprx}, we first approximate the continuous time partially observable diffusion model with a discrete time model. That approximation, however, is obtained by sampling the true diffusion, and the resulting model is not a standard discrete-time partially observable Markov decision process (POMDP): In a POMDP, the measurement at each stage must be generated from the current state through a kernel $Q(dy \mid x)$ (possibly also with control dependence); in the sampled model of \cite{pradhanyuksel2023DTApprx} the measurement carried by the interval $[kh,(k+1)h)$ depends on the entire state trajectory over that interval, so no such kernel exists and, given the sampled state, the measurement is not conditionally independent of the past. Consequently the standard discrete-time POMDP machinery cannot be invoked for that model.
    We resolve this by using an approximation which is of the Euler--Maruyama type given above such that when the current hidden state is given, the (current) measurement is independent of history.
    \item \textbf{Finite Model and Filter/Belief Quantization.} The Euler--Maruyama discretization (\ref{eq:euler_state0})-(\ref{eq:euler_obs0}) gives rise to a discrete time partially observable Markov decision process (POMDP) with mathematically convenient regularity structure.
   % \item 
    %the state process. We provide strong error bounds for this approximation (Theorem \ref{thm:strong_convergence_mismatch}) and analyze the convergence of the Radon--Nikodym derivatives involved in the cost functional.
 Towards this direction, to enable numerical computation with rigorous performance guarantees, we establish the weak Feller property of the filter kernel for the time-discretized model (\ref{eq:euler_state0})-(\ref{eq:euler_obs0}). This regularity allows us to apply quantization techniques to the probability/belief space. We also show that observation quantization gives a near optimal policy.
    \item \textbf{Near Optimality of Policies from Finite POMDP obtained from time, space and measurement discretization.} The central result of this work is an approximation analysis showing that near optimality is preserved across this entire chain of approximations. We prove that a policy derived from the discretized, quantized model is near optimal for the original continuous-time controlled diffusion.
\end{enumerate}

The remainder of the paper is organized as follows. Section~\ref{sec:cont_model} introduces the continuous time model and wide-sense admissible policies, details the discrete-time approximation of \cite{pradhanyuksel2023DTApprx} and presents the Euler-Maruyama scheme. Section~\ref{numerical} presents the main approximation results. Finally, Section~\ref{asymptotic_continous} studies regularity properties and asymptotic optimality of quantized models, utilizing the supporting results of Section~\ref{numerical}.

\section{Continuous Time Partially Observed Control Model and Existence of Optimal Solutions}\label{sec:cont_model}

We consider the following continuous time controlled diffusion process which solution is defined on a complete probability space $(\Omega, \mathcal{F}, P)$:
\begin{equation}\label{cont_model}
dX_t = b(X_t, U_t) \, dt + \sigma(X_t) \, dW_t, \quad X_0 = x \in \mathbb{R}^d,
\end{equation}
where \( X_t \in \mathbb{R}^d \) is the system state, \( U_t \in \mathbb{U} \) is the control input (with \( \mathbb{U} \) compact), \( W_t \) is a $d$-dimensional standard Wiener process, $b: R^d \times \mathbb{U} \to R^d$, and $\sigma: \mathbb{R}^d \to  \mathbb{R}^{d\times d} $ are given functions.

\begin{assumption}\label{cont_bdd}
\textbf{Continuity and Boundedness.} The drift function \( b \) is jointly continuous in \( (x, u) \), and the diffusion coefficient \( \sigma \) is locally Lipschitz continuous. That is, for every \( R > 0 \), there exists a constant \( C_R > 0 \) such that
\[
\left\|\sigma(x_1) - \sigma(x_2)\right\|^2 \leq C_R \left|x_1 - x_2\right|^2, \quad \forall x_1, x_2 \in \mathbb{B}_R, \, u \in \mathbb{U},
\]
where \( \|\sigma\| := \sqrt{\operatorname{Tr}(\sigma \sigma^\top)} \) and $\mathbb{B}_R := \{x \in \mathbb{R}^d : |x| < R\}$ denotes the open ball of radius $R$ in $\mathbb{R}^d$, centered at the origin.
Additionally, we assume that both \( b \) and \( \sigma \) are uniformly bounded, i.e,  there exists a constant \( K> 0 \) such that
\[
\sup_{u \in \mathbb{U}} |b(x, u)| + \|\sigma(x)\| \leq K, \quad \forall x \in \mathbb{R}^d.
\]
\end{assumption}

\begin{assumption}\label{nondeg}
\textbf{Nondegeneracy.} There exists a constant \( \hat{C}_1 > 0 \) such that for all \( x, z \in \mathbb{R}^d \) ,
\[
\sum_{i,j=1}^d a^{ij}(x) z_i z_j \geq \hat{C}_1 |z|^2,
\]
where \( a := \frac{1}{2} \sigma \sigma^\top \).
\end{assumption}
We utilize this set of assumptions to ensure existence and uniqueness of a weak solution to (\ref{cont_model}); see \cite{BoArGh12}.

The cost associated with a control policy is given by the finite horizon objective:
\[
J_T(x, U) := \mathbb{E}_x^U \left[ \int_0^T c(X_t, U_t) \, dt + c_T(X_T) \right],
\]
where \( c : \mathbb{R}^d \times \mathbb{U} \to \mathbb{R}_+ \) and \( c_T : \mathbb{R}^d \to \mathbb{R}_+ \) are bounded and continuous.

However, the signal process \( X_t \) is not directly observable. Instead, we observe:
\begin{align}\label{obs_model}
dY_t = g(X_t) \, dt + dB_t,
\end{align}
where \( g : \mathbb{R}^d \to \mathbb{R}^M \) is continuous and bounded, and \( B_t \) is a \( \mathbb{R}^M \)-valued Wiener process independent of \(W_t\).

The control process \( \{U_t\} \) must be adapted to the filtration generated by the observation process \( \{Y_t\} \); that is, each \( U_t \) must be \( \sigma(Y_s : 0 \leq s \leq t) \)-measurable. These are referred to as \emph{strict-sense admissible policies}.
To establish the existence of optimal controls, Fleming and Pardoux \cite{FlPa82} introduced a broader class of policies, called \emph{wide-sense admissible policies}. The idea is to apply Girsanov's transformation so that the observation process becomes an independent Wiener process under a new probability measure \(  \mathbb{P}_{ref} \). 
Following \cite[Section 3.2]{FlPa82}, admissible controls are defined as probability measures on \( C([0, T]; \mathbb{R}^M) \times \mathcal{P}_\lambda([0, T] \times \mathbb{U}) \) (where $\mathcal{P}_\lambda([0, T] \times \mathbb{U})$ is the space of positive finite measures on $[0,T]\times \mathbb{U}$ with its fixed marginal on $[0,T]$ to be the Lebesgue measure $\lambda$), with the first marginal being the Wiener measure. Furthermore, under measure \( \mathbb{P}_{ref} \), the increments \( Y_r - Y_t \) are independent of \( \{X_0, W_{[0,t]}, Y_{[0,t]}, m_{[0,t]}\} \) for all \( 0 \leq t \leq r \leq T \).

This means that, rather than requiring controls to depend on past observations, we enforce that future observations be conditionally independent of past controls and observations. Denote the set of wide-sense admissible policies by \( \mathfrak{U}_{\mathrm{WS}} \), equipped with the weak topology. Here the \emph{weak topology} on $\mathfrak{U}_{\mathrm{WS}}$ is the topology of weak convergence of probability measures on $C([0,T];\mathbb{R}^M)\times\cP_\lambda([0,T]\times\U)$, following \cite[Section~3.2]{FlPa82} and \cite[Section~2]{pradhanyuksel2023DTApprx}: a sequence $m_n\to m$ if and only if $\int f\,dm_n\to \int f\,dm$ for every bounded continuous function $f$ on this product space. Under this topology $\mathfrak{U}_{\mathrm{WS}}$ is compact, since the first marginal is fixed to be the Wiener measure, the second marginal has fixed time-marginal $\lambda$ on $[0,T]$ and takes values in the compact set $\U$, and Prohorov's theorem applies (see \cite[Theorem~3.1]{FlPa82} and \cite[Lemma~3.3.1]{BoArGh12}). Typical elements are denoted by $m$.

Under the new measure \( \mathbb{P}_{ref} \), we have:
\begin{align}
dX_t &= b(X_t, U_t) \, dt + \sigma(X_t) \, dW_t, \\
dY_t &= dB_t,
\end{align}
which transforms the problem into one with independent observations. 

We have $\mathbb{P} \ll \mathbb{P}_{ref}$, and the Radon--Nikodym derivative is
\[
\frac{d \mathbb{P}}{d \mathbb{P}_{ref}} = G_T
\]
where
\[
G_t := \exp\left( \int_0^t g(X_s) \, dY_s - \frac{1}{2} \int_0^t |g(X_s)|^2 \, ds \right),
\]
and the stochastic integral is with respect to \( Y_s \). Since \( g \) is bounded, Novikov's condition \cite[Corollary~3.5.13]{KaSh91} ensures that \( \{G_t\} \) is a nonnegative martingale with mean one. It ensures that it is a well-defined Radon-Nikodym derivative. This relation allows us to view the partially observable problem as one with independent
measurements, with the dependence pushed to the Radon-Nikodym derivative $G_T$.

\subsection{Existence of Optimal Policies and Discrete-Time Approximations}\label{sec:cont_existence_disc}

\cite{pradhanyuksel2023DTApprx} states the following theorem which essentially proves continuity of the value function over wide-sense admissible policies, existence of optimal policy in wide-sense admissible class, and near optimality of piecewise-constant wide-sense policies.

\begin{theorem}\cite[Theorem 3.1]{pradhanyuksel2023DTApprx}\label{thm:PY_existence}
Suppose Assumptions \ref{cont_bdd} and \ref{nondeg} hold. Then:
\begin{itemize}
  \item[(i)] The total cost
  \[
  J_T(x, m) = \mathbb{E}_{x, \mathbb{P}_{ref}}^m\left[ \frac{d\mathbb{P}}{d \mathbb{P}_{ref}} \left( \int_0^T c(X_t, m_t) \, dt + c_T(X_T) \right) \right]
  \]
  is continuous over the space of wide-sense admissible policies \( \mathfrak{U}_{\mathrm{WS}} \) .

  \item[(ii)] There exists an optimal policy in \( \mathfrak{U}_{\mathrm{WS}} \).

  \item[(iii)] For every \( \varepsilon > 0 \), there exists a piecewise-constant control policy in \( \mathfrak{U}_{\mathrm{WS}} \) that is \( \varepsilon \)-optimal.
\end{itemize}
\end{theorem}

\cite{pradhanyuksel2023DTApprx} discretizes time for the controlled diffusion as follows, with piecewise-constant control policy and step size $h=1/n$:
\begin{align}\label{disc_policy}
X^h_{t} &= X^h_{0} + \int_{0}^{t} b(X^h_s, U^h_{s}) \, ds + \int_{0}^{t} \sigma(X^h_s) \, dW_s, \\
Y^h_t &= Y^h_{0} + \int_{0}^t g(X^h_s) \, ds + \int_{0}^t dB_s
\end{align}
with initial condition \( X_0 = x \),
where
\[
U_s^h := U_{\eta_h(s)}, \quad \eta_h(s):=  h \lfloor s/h \rfloor
\]
where $\lfloor \cdot \rfloor$ denotes the floor function. That is $X^h$ is the true controlled diffusion driven by a control that is constant on each interval $[kh, (k+1)h)$.

Now we can discretize signal and path valued observation with
\[
\hat{X}_k^h := X_{kh}^h, \quad \hat{U}_k^h := U^h_{kh}, \quad \hat{Y}_k^h := Y^h_{[kh, (k+1)h)}
\]
and, with \( N_h := \lfloor T/h \rfloor \), the discrete-time cost becomes
\begin{align}\label{def:new_value_fn}
J_h(x, U^h) = \mathbb{E}_x^{U^h} \left[ \sum_{k=0}^{N_h - 1} \hat{c}_h(X_{kh}^h, U_{kh}) + c_T(X_{N_h h}^h) \right],
\end{align}
where:
\begin{align}\label{def:new_cost_fn}
\hat{c}_h(x, u) := \mathbb{E}_x \left[ \int_0^h c(X_s, u) \, ds \,\bigg|\, X_0 = x \right],
\end{align}
here $X_s$ is the signal at time $s$, with initial condition $X_0=x$ and under constant action $u$.

This defines a discrete-time process with path-valued observations \( \hat{Y}_k^h \), and cost \( \hat{c}_h \).

\begin{theorem}\cite[Theorem 5.2]{pradhanyuksel2023DTApprx}\label{PY_pathvalued}
Under Assumptions \ref{cont_bdd} and \ref{nondeg}, as \( h \to 0 \)  (equivalently $n \to \infty$ ) :

\begin{enumerate}
    \item The optimal value of the discrete-time approximation converges to that of the original continuous-time control problem.
\item Moreover, for any $\varepsilon>0$, there exists a sufficiently small $h>0$ such that the policy obtained from the discrete-time model is $\varepsilon$-optimal for the original continuous-time setting.
\end{enumerate}
\end{theorem}

% While this result provides near-optimal policies via discretization, it assumes access to path-valued measurements. However, we want to discritize observation also, we only observe \( Y_{kh} \) in the original measurement space. Thus, we seek a discrete-time model where \( U_{kh} = \gamma(Y_{ih} : i = 0, \ldots, k) \) under an admissible policy \( \gamma \), which is a measurable map taking values in \( \mathcal{P}(\mathbb{U}) \). The following result addresses this issue.

% \begin{align}\label{apprx_system}
% X_{(k+1)h} &= X_{kh} + \int_{kh}^{(k+1)h} b(X_s, U_{kh}) \, ds + \int_{kh}^{(k+1)h} \sigma(X_s) \, dW_s, \nonumber \\
% Y_t &= Y_{kh} + g(X_{kh}) + \int_{kh}^t dB_s, \quad t \in [kh, (k+1)h).
% \end{align}

% \begin{theorem}\cite[Theorem 5.3]{pradhanyuksel2023DTApprx}
% Suppose Assumptions A1 and A2 hold. Then:
% \begin{itemize}
%     \item[(i)] The optimal value of the discrete-time model (\ref{apprx_system}) converges to that of the original continuous-time model.
%     \item[(ii)] For every \( \varepsilon > 0 \), there exists \( h > 0 \) such that the discrete-time policy
%     \[
%     U_{kh}^\varepsilon = \gamma^\varepsilon(Y_{ih} : i = 0, \ldots, k)
%     \]
%     constructed via (\ref{apprx_system}) is \( \varepsilon \)-optimal for the original problem. 
% \end{itemize}
% \end{theorem}

While the preceding result provides near-optimal policies via discretization, it fundamentally assumes that the controller has access to continuous path-valued measurements over each interval. A more practical objective is to formulate a model that utilizes only discrete-time sampled measurements, i.e., $Y_{kh}$. Specifically, at time $kh$, the control $U_{kh}$ should depend solely on the sampled history $\{Y_{ih} : i=0, \dots, k\}$ via an admissible policy $\gamma$.

By employing Lusin's theorem\footnote{\cite[Theorem~7.1.13]{Bog07}: let $(\Omega,\mathcal{F},\mu)$ be a probability space with $\Omega$ Polish, $(\U,d)$ a separable metric space and $\gamma:\Omega\to\U$ measurable. For every $\epsilon>0$ there exists a compact set $\Omega_\epsilon\subset\Omega$ with $\mu(\Omega_\epsilon)>1-\epsilon$ such that the restriction $\gamma|_{\Omega_\epsilon}$ is continuous.} and measure transformation techniques (see \cite[Section 5.2]{pradhanyuksel2023DTApprx}), we can establish the near optimality of such sampled policies. Applying the Girsanov transformation to \eqref{disc_policy} changes the system to the reference measure $\Pref$ under which the observation process $Y$ becomes a standard Brownian motion independent of the state:
\begin{align}\label{DiscretePOMSomSIAM}
X_{(k+1) h} &= X_{k h} + \int_{k h}^{(k+1) h} b\left(X_s, U_{k h}\right) ds + \int_{k h}^{(k+1) h} \sigma\left(X_s\right) dW_s, \\
Y_t &= Y_{k h} + \int_{k h}^t dB_s, \quad t \in[k h,(k+1) h).
\end{align}

Under the reference measure $\Pref$, wide-sense policies are generated by measurable mappings on the observation path space: at each stage $k$, the (relaxed) control action is $\gamma_k(Y_{[0,kh]})$, where $\gamma_k : C([0,kh]; \R^M) \to \cP(\U)$ is a measurable mapping. $\cP(\U)$, equipped with the weak convergence topology, is a convex and compact metric space. Under $\Pref$, the law of $Y_{[0,kh]}$ is the Wiener measure on the Polish space $C([0,kh]; \R^M)$, so by Lusin's theorem, for every $\epsilon > 0$ there exists a compact set of paths $\Omega^k_\epsilon \subset C([0,kh]; \R^M)$ with $\Pref(Y_{[0,kh]} \in \Omega^k_\epsilon) > 1 - \epsilon/N_h$ such that the restriction of $\gamma_k$ to $\Omega^k_\epsilon$ is continuous, where $N_h = \lfloor T/h \rfloor$ is the number of control stages. Taking the intersection of the corresponding events over all stages $k$, which we denote by $\Omega_\epsilon$ below, we obtain an event on which this continuity holds for every stage simultaneously. Since $\cP(\U)$ is a convex compact set, by the Dugundji extension theorem\footnote{\cite[Theorem~4.1]{dugundji1951}: let $X$ be a metric space, $A \subset X$ a closed subset, $L$ a locally convex linear space, and $f : A \to L$ a continuous mapping. Then $f$ admits a continuous extension $F : X \to L$ such that $F(X)$ is contained in the convex hull of $f(A)$; this generalizes the Tietze extension theorem \cite[Theorem~35.1]{Munkres}.}, this restriction can be extended to a continuous mapping $\gamma_{C,k} : C([0,kh]; \R^M) \to \cP(\U)$ on the entire path space, with $\gamma_{C,k} = \gamma_k$ on $\Omega^k_\epsilon$; see also \cite{saldiyukselGeoInfoStructure,pradhanyuksel2023DTApprx}. We write $\gamma_C = (\gamma_{C,k})_k$ for the resulting policy, and
\[
\Omega_\epsilon := \bigcap_{k=0}^{N_h - 1} \left\{ Y_{[0,kh]} \in \Omega^k_\epsilon \right\}, \qquad \Pref(\Omega_\epsilon) > 1 - \epsilon,
\]
for the event on which the two policies produce exactly the same control actions.

Now, to quantify the performance loss of replacing the true measurable policy $\gamma$ with its continuous approximation $\gamma_C$, consider the cost functional under the measure $\Pref$, which includes the Radon--Nikodym derivative $G_T = \frac{d\mathbb{P}}{d\Pref}$. The key observation is that $\Omega_\epsilon$ is an event of the observation path only, and under $\Pref$ the observation path is a Brownian motion whose law does not depend on the policy; therefore $\Omega_\epsilon$ is the same event, with the same probability, under both policies. On $\Omega_\epsilon$ the two policies produce the same control paths, hence the same state paths and the same Radon--Nikodym derivatives, so any difference in cost arises on the complement $\Omega_\epsilon^c$. Since the stage cost $c$ and the terminal cost $c_T$ are uniformly bounded, the difference in expected costs is strictly bounded by
\[
|J_h(\gamma) - J_h(\gamma_C)| \le M\, \E_{\Pref}\left[ \left( G^{\gamma}_T + G^{\gamma_C}_T \right) \mathbf{1}_{\Omega_\epsilon^c} \right] \le 2 M e^{\|g\|_\infty^2 T / 2} \sqrt{\epsilon},
\]
where $M = T\|c\|_\infty + \|c_T\|_\infty$, and $G^{\gamma}_T$ and $G^{\gamma_C}_T$ denote the Radon--Nikodym derivatives under the two policies. For the last inequality, we use the Cauchy--Schwarz inequality together with the following second moment bound: since $g$ is bounded,
\begin{equation}\label{eq:second_moment}
\E_{\Pref}\left[ G_T^2 \right] \le e^{\|g\|_\infty^2 T}.
\end{equation}
To see this, write
\[
G_T^2 = \exp\left( \int_0^T 2 g(X_t) \, dY_t - \frac{1}{2}\int_0^T |2 g(X_t)|^2 \, dt \right) \exp\left( \int_0^T |g(X_t)|^2 \, dt \right)
\le e^{\|g\|_\infty^2 T}\, \mathcal{E}_T,
\]
where $\mathcal{E}_T$ is the exponential martingale associated with the integrand $2g(X_t)$, whose $\Pref$-expectation is one, again by Novikov's condition.
Then, for either policy,
\[
\E_{\Pref}\left[ G_T \mathbf{1}_{\Omega_\epsilon^c} \right]
\le \left(\E_{\Pref}\left[ G_T^2 \right]\right)^{1/2} \sqrt{\Pref(\Omega_\epsilon^c)}
\le e^{\|g\|_\infty^2 T/2} \sqrt{\epsilon}.
\]
Letting $\epsilon \to 0$, we see that $\gamma$ can be approximated by $\gamma_C$ with an arbitrarily small loss in performance (see also \cite[Proposition 6.1]{saldiyukselGeoInfoStructure}).

Next, we evaluate this continuous policy $\gamma_C$ on the discrete measurements. Under $\Pref$, $Y_t$ is a standard Brownian motion, meaning its sample paths are almost surely uniformly continuous on the compact interval $[0, T]$. Consequently, the piecewise-constant version of the discrete measurements, defined as $\hat{Y}_t^{h} := Y_{ih}$ for $t \in [ih, ih+h)$, converges uniformly to the continuous path $Y_t$ as $h \to 0$; that is, $\sup_{t \in [0,T]} |\hat{Y}_t^{h} - Y_t| \to 0$ almost surely. Because each $\gamma_{C,k}$ is a continuous mapping on the path space, evaluating it on the discretized sample path yields $\gamma_{C,k}(\hat{Y}^h_{[0,kh]}) \to \gamma_{C,k}(Y_{[0,kh]})$ almost surely, and the policy on the left uses only the finitely many samples $\{Y_{ih} : i \le k\}$, as desired. Almost sure convergence of the actions implies that the joint law of the observation path and the induced relaxed control converges in the weak topology of $\mathfrak{U}_{\mathrm{WS}}$, so by the continuity of the cost in the policy (Theorem~\ref{thm:PY_existence}(i)) the difference in expected cost between evaluating $\gamma_C$ on the full continuous path and the discretely sampled path vanishes as $h \to 0$.

This directly leads to the following refinement for discrete-time approximations, for both the finite-horizon and the discounted cost criteria.

% \begin{theorem}
% Suppose Assumptions (A1) and (A2) hold. Then:
% \begin{enumerate}
%     \item[(i)] The optimal value of the discrete-time model converges to the optimal value of the original continuous-time model.
%     \item[(ii)] For every $\epsilon > 0$, there exists a discretization step $h > 0$ such that the policy obtained from the discrete-time approximation is near-optimal for the original continuous-time system.
% \end{enumerate}
% \end{theorem}

% However, it need to access path-valued observations. And We can similarly establish the following result for the case where only discrete-time samples of the observation process are used (instead of full path-valued observations). In particular:
% \[
% dX_t = b(X_t, U_t) \, dt + \sigma(X_t) \, dW_t, \quad
% dY_t = g(X_t) \, dt + dB_t,
% \]
% with the same initial condition and control structure, and define the evolution approximately (Model 3) as:
% \[
% \begin{aligned}
% X^h_{(k+1)h} &= X^h_{kh} + \int_{kh}^{kh+h} b(X_s, U_{kh}) ds + \int_{kh}^{kh+h} \sigma(X_{s}) W_{s} , \\
% Y^h_{(k+1)h} &= Y^h_{kh} + g(X^h_{kh+h}) h + (B_{(k+1)h} - B_{kh}).
% \end{aligned}
% \]

% Now one can define discrete time structure with the following components:
% \[
% X_k^h := X_{kh}^h, \quad U_k^h := U_{kh}, \quad \bar{Y}_k^h := Y^h_{(k+1)h} - Y^h_{kh}.
% \]
% For this new model we have following result:
\begin{theorem}\cite[Theorem 5.3]{pradhanyuksel2023DTApprx}
Under Assumptions \ref{cont_bdd} and \ref{nondeg}:

\begin{enumerate}
\item The optimal value of the discrete-time model  convergences to the optimal value of the original continuous-time model.
\item For every $\epsilon>0$, there exists $h>0$ so that the solution of the discrete-time approximation gives a policy (i.e., a policy $U_{k h}^\epsilon=\gamma^\epsilon\left(Y_{i h}, i \in\{0,1, \cdots, k\}\right)$) which is near optimal for the original continuous-time model.
\end{enumerate}
\end{theorem}

\section{Main Results}\label{numerical}

In this section, we present our main results: a discrete-time Euler--Maruyama approximation, regularity properties of the resulting discrete-time POMDP, and quantization-based near optimality.

\subsection{Discrete-Time Euler--Maruyama Approximation}\label{sec:cont_em}

The discrete-time model given in Section~\ref{sec:cont_existence_disc} (equation~\eqref{disc_policy}) suffers from the following: in a typical POMDP (see Section~\ref{sec:cont_reg}), given the current state, the current measurement is conditionally independent of the past states, measurements, and actions. This does not hold under equation~\eqref{disc_policy}. To see this, observe that the model in \eqref{disc_policy} is a sampled version of the true SDE starting at $X_0$, driven by a piecewise-constant policy $U_t = U_{kh}$ for $t \in [kh, (k+1)h)$:
\begin{align}
    X_{kh}^h &= X_{kh-h}^h + \int_{kh-h}^{kh} b(X^h_s, U_{kh-h}) \, ds + \int_{kh-h}^{kh} \sigma(X^h_s) \, dW_s, \label{eq:sde_true}\\
    Y_{kh}^h &= Y_{kh-h}^h + \int_{kh-h}^{kh} g(X_s^h) \, ds + \int_{kh-h}^{kh} dB_s. \label{eq:sde_true_obs}
\end{align}
However, even after defining the measurement to be the increment $Y^h_{(k+1)h}-Y^h_{kh}$, the desired property does not hold: the increment equals $\int_{kh}^{(k+1)h} g(X^h_s)\,ds + (B_{(k+1)h}-B_{kh})$ and thus depends on the entire state trajectory over the interval $[kh,(k+1)h]$, not only on the sampled state $X^h_{kh}$. Consequently, given the current sampled state, the measurement is not conditionally independent of the past, and no measurement kernel of the form $Q(dy \mid x)$ exists for the sampled model. 

To resolve this, we construct an Euler--Maruyama discretization utilizing a similar observation structure, and evaluate the discretization error between the piecewise-constant diffusion \eqref{disc_policy} and its Euler--Maruyama counterpart. For the rest of the paper, in addition to Assumptions \ref{cont_bdd} and \ref{nondeg}, we assume the following.

\begin{assumption}\label{assump:lipschitz}
There exists a constant $K > 0$ (without loss of generality, the same constant as in Assumption \ref{cont_bdd}) such that for all $x, y \in \R^d$ and $u \in \U$:
    \[
        |b(x,u) - b(y,u)| + \|\sigma(x) - \sigma(y)\| \leq K |x - y|.
    \]
\end{assumption}

The Euler--Maruyama discretization is defined via the grid points $t_i = ih$. Let $X^{EM}_0 = X_0$ and $Y^{EM}_0 = Y_0$. The system evolves for $k = 1, 2, \dots$ as
\begin{align}
X^{EM}_{kh} &= X^{EM}_{kh-h} + \int_{kh-h}^{kh} b(X^{EM}_{kh-h}, U_{kh-h}) \, ds + \int_{kh-h}^{kh} \sigma(X^{EM}_{kh-h}) \, dW_s, \label{eq:euler_state} \\
Y^{EM}_{kh} &= Y^{EM}_{kh-h} + \int_{kh-h}^{kh} g(X^{EM}_{kh-h}) \, ds + \int_{kh-h}^{kh} dB_s.\label{eq:euler_obs}
\end{align}

For policy admissibility, realize that under $\Pref$ the observation increments for both models are simply $B_{(k+1)h} - B_{kh}$, which are independent Gaussian random variables. Because the observation process $Y$ is entirely decoupled from the state dynamics under $\Pref$, any wide-sense admissible policy adapted to the filtration generated by $Y$ is strictly applicable to both models.

To analyze the approximation error, we define the continuous-time version (interpolation) of the Euler--Maruyama scheme. Recall that $\eta_h(s) = h \lfloor s/h \rfloor$; the continuous-time process is
\begin{equation} \label{eq:euler_diff_init}
    X^{EM}_{t} = X_0 + \int_{0}^{t}  b(X^{EM}_{\eta_h(s)}, U_{\eta_h(s)}) \, ds + \int_{0}^{t}  \sigma(X^{EM}_{\eta_h(s)}) \, dW_s.
\end{equation}

The following theorem establishes the convergence of the state processes, uniform over all piecewise-constant admissible policies.

\begin{theorem}\label{thm:strong_convergence_mismatch}
Suppose Assumptions \ref{cont_bdd} and \ref{assump:lipschitz} hold, and fix a finite horizon $T > 0$. Consider the sampled diffusion \eqref{disc_policy} and its Euler--Maruyama interpolation \eqref{eq:euler_diff_init} (the continuous-time versions of \eqref{eq:sde_true} and \eqref{eq:euler_state}, respectively), both starting from the same initial condition $X_0$ and driven by the same Wiener process $W$ and the same open loop piecewise-constant control process $U_{\eta_h(\cdot)}$. Then
\[
    \E\left[ \sup_{0 \leq t \leq T} |X_t^{EM} - X_t^h|^2 \right] \leq  C_2 T e^{2C_1T},
\]
where $C_1=2(T+4)K^2$ and $C_2=4C_1K^2(h^2+h)$; hence $C_2 \to 0$ as the step size $h \to 0$. 
\end{theorem}

\begin{proof}
Throughout, $U$ may be any $\U$-valued process adapted to a filtration with respect to which $W$ is a Wiener process; in particular, any wide-sense admissible policy applied under $\Pref$ qualifies. Let $Z_t = X_t^h- X_t^{EM}$. We omit $U_{\eta_h(s)}$ from the notation for simplicity; all estimates below are uniform in the control since the constants in Assumptions \ref{cont_bdd} and \ref{assump:lipschitz} are uniform over $\U$. Subtracting \eqref{eq:euler_diff_init} from \eqref{disc_policy}:
\[
    Z_t = \int_0^t \left( b(X^h_s) - b(X_{\eta_h(s)}^{EM}) \right) \, ds + \int_0^t \left( \sigma(X_s^{h}) - \sigma(X_{\eta_h(s)}^{EM}) \right) \, dW_s.
\]
Since $b$ and $\sigma$ are bounded, all moments appearing below are finite. We aim to bound $\E[\sup_{0 \leq r \leq t} |Z_r|^2]$. Using the inequality $|a + b |^2 \leq 2|a|^2 + 2|b|^2$, we have
\[
    |Z_t|^2 \leq 2 \left| \int_0^t (b(X^h_s) - b(X_{\eta_h(s)}^{EM})) \, ds \right|^2 + 2\left| \int_0^t (\sigma(X_s^{h}) - \sigma(X_{\eta_h(s)}^{EM})) \, dW_s \right|^2.
\]
Therefore, taking the supremum and expectation, we have
\begin{align} \label{eq:error_split_3_sup}
    \E[\sup_{0 \leq r \leq t} |Z_r|^2] \;\leq\;&  2 \E\!\left[ \sup_{0 \leq r \leq t} \!\left| \int_0^r (b(X^h_s) - b(X_{\eta_h(s)}^{EM})) \, ds \right|^2 \right]\nonumber\\
    &\;+\;2 \E\!\left[ \sup_{0 \leq r \leq t}\! \left| \int_0^r (\sigma(X^h_s) - \sigma(X_{\eta_h(s)}^{EM})) \, dW_s \right|^2 \right].
\end{align}
Using the Cauchy--Schwarz inequality on the drift term and the Burkholder--Davis--Gundy (BDG) inequality \cite[Theorem~3.3.28]{KaSh91} (which gives, for a continuous local martingale $M$ with quadratic variation $\langle M\rangle$, the bound $\E[\sup_{0\le r\le t}|M_r|^2]\le 4\,\E[\langle M\rangle_t]$) on the diffusion term, together with the Lipschitz property of $b$ and $\sigma$, we obtain
\begin{align*}
    \E\!\left[ \sup_{0 \leq r \leq t} \left| \int_0^r (b(X_s^h) - b(X_{\eta_h(s)}^{EM})) \, ds \right|^2 \right] &\leq t K^2 \int_0^t \E|X_s^h - X_{\eta_h(s)}^{EM}|^2 \, ds, \\
    \E\!\left[ \sup_{0 \leq r \leq t} \left| \int_0^r (\sigma(X_s^h) - \sigma(X_{\eta_h(s)}^{EM})) \, dW_s \right|^2 \right] &\leq 4 K^2 \int_0^t \E|X_s^h - X_{\eta_h(s)}^{EM}|^2 \, ds.
\end{align*}
For the first term, by Cauchy--Schwarz we have $\bigl(\int_0^r f(s)\,ds\bigr)^2 \le r\int_0^r f(s)^2\,ds \le t \int_0^t f(s)^2\, ds$ for $r \le t$, and by the Lipschitz property of $b$ we have $|b(x)-b(y)|^2\le K^2 |x-y|^2$. For the second term, BDG is applied to the local martingale $M_r = \int_0^r (\sigma(X_s^h)-\sigma(X^{EM}_{\eta_h(s)}))\,dW_s$, whose quadratic variation $\langle M\rangle_t = \int_0^t \|\sigma(X_s^h)-\sigma(X^{EM}_{\eta_h(s)})\|^2\,ds$ is bounded by $K^2 \int_0^t |X_s^h - X^{EM}_{\eta_h(s)}|^2\, ds$ by the Lipschitz property of $\sigma$.

Substituting these back into \eqref{eq:error_split_3_sup}, we have
\[
    \E\left[ \sup_{0 \leq r \leq t} |Z_r|^2 \right]
    \leq  2(t+4)K^2 \int_0^t \E|X_s^h - X_{\eta_h(s)}^{EM}|^2 \, ds
    \leq C_1 \int_0^t \E|X_s^h - X_{\eta_h(s)}^{EM}|^2 \, ds,
\]
where $C_1= 2(T+4)K^2$. We split the error term inside the integral:
\[
    |X_s^h - X_{\eta_h(s)}^{EM}|^2 \leq 2|X_s^h - X_s^{EM}|^2 + 2|X_s^{EM} - X_{\eta_h(s)}^{EM}|^2.
\]
Recalling that $Z_s=X_s^h- X_s^{EM}$ and that $\E|Z_s|^2 \le \E[\sup_{0\le r\le s}|Z_r|^2]$, we obtain
\begin{equation} \label{eq:gronwall_setup}
    \E\left[ \sup_{0 \leq r \leq t} |Z_r|^2 \right] \leq   2C_1 \int_0^t \E\left[\sup_{0 \le r\le s}|Z_r|^2\right] ds
    + 2C_1 \int_0^t \E|X_s^{EM} - X_{\eta_h(s)}^{EM}|^2 \, ds .
\end{equation}
The term $\E|X_s^{EM} - X_{\eta_h(s)}^{EM}|^2$ represents the increment of the Euler scheme within a partial step. Applying $|a+b|^2\le 2|a|^2+2|b|^2$ to the increment $X_s^{EM}-X^{EM}_{\eta_h(s)} = b(X^{EM}_{\eta_h(s)})(s-\eta_h(s)) + \sigma(X^{EM}_{\eta_h(s)})(W_s-W_{\eta_h(s)})$, using the uniform bound  $\sup_{u \in \U} |b(x, u)| + \|\sigma(x)\| \le K$, and noting that $\sigma(X^{EM}_{\eta_h(s)})$ is measurable with respect to the past while $W_s-W_{\eta_h(s)}$ is an independent increment, we get
\begin{align*}
     \E|X_s^{EM} - X_{\eta_h(s)}^{EM}|^2 &\leq 2 \E| b (X^{EM}_{\eta_h(s)}) (s-\eta_h(s)) |^2 + 2 \E| \sigma (X^{EM}_{\eta_h(s)}) (W_s - W_{\eta_h(s)}) |^2 \\
     &\leq 2 K^2 h^2 + 2 K^2 h.
\end{align*}

Let $\phi(t) = \E[\sup_{0 \leq r \leq t} |Z_r|^2]$. Equation \eqref{eq:gronwall_setup} combined with the bound $\E|X_s^{EM}-X^{EM}_{\eta_h(s)}|^2 \le 2K^2(h^2+h)$  yields
\[
    \phi(t) \leq C_2 t  + 2C_1 \int_0^t \phi(s) \, ds,
\]
where $C_2 = 4 C_1 K^2 (h^2 + h)$. By Gr\"onwall's lemma\footnote{\cite[Problem~5.2.7]{KaSh91}: if $\phi:[0,T]\to\R_+$ is integrable and satisfies $\phi(t)\le a(t) + b\int_0^t \phi(s)\,ds$ for some constant $b\ge 0$ and nondecreasing $a(\cdot)\ge 0$, then $\phi(t)\le a(t)\,e^{bt}$ for all $t\in[0,T]$.} with $a(t)=C_2 t$ and $b=2C_1$, we have
\[
    \phi(t) \leq  C_2 t e^{2C_1 t}.
\]
Setting $t=T$ yields the claimed bound. Note that $C_2 = 4C_1 K^2(h^2+h) = O(h)$ as $h\to 0$, so the strong error $\left(\E\left[ \sup_{0 \leq t \leq T} |X_t^{EM} - X_t^h|^2 \right]\right)^{1/2}$ vanishes at rate $\sqrt{h}$.
\end{proof}

We now present the main approximation theorem: we establish that optimizing the computable Euler--Maruyama POMDP gives a policy that is asymptotically near optimal for the true diffusion model. Let $J_h(\gamma)$ denote the expected cost under the sampled diffusion model, and let $J^{EM}_h(\gamma)$ denote the expected cost under the fully discrete Euler--Maruyama POMDP. Utilizing the discrete-time stage cost $\hat{c}_h$ of \eqref{def:new_cost_fn}, these costs are evaluated as
\begin{align*}
    J_h(\gamma) &= \E_{\mathbb{P}}^{\gamma} \left[ \sum_{k=0}^{N_h - 1} \hat{c}_h(X^h_{kh}, U_{kh}) + c_T(X^h_{N_h h}) \right], \\
    J^{EM}_h(\gamma) &= \E_{\mathbb{P}}^{\gamma} \left[ \sum_{k=0}^{N_h - 1} \hat{c}_h(X^{EM}_{kh}, U_{kh}) + c_T(X^{EM}_{N_h h}) \right],
\end{align*}
where $N_h = \lfloor T/h \rfloor$. (Without loss of generality, we can assume that $N_h =  T/h$)

Girsanov's theorem cannot be applied directly to the fully discrete model; a piecewise-constant observation process does not admit a valid Radon--Nikodym derivative with respect to the continuous-time Wiener measure. Therefore, we proceed in two steps. First, we define an intermediate continuous-observation Euler--Maruyama model for which Girsanov's measure transformation is valid, and we bound the error between this intermediate model and the exact sampled diffusion. Second, we use Lusin's theorem to prove that restricting the policy to depend strictly on discrete observation samples incurs no asymptotic performance loss.

We define the intermediate model so that the state $X^{EM}_t$ evolves according to the continuous-time Euler--Maruyama interpolation \eqref{eq:euler_diff_init}, while the observation process $\tilde{Y}_t$ remains a continuous diffusion driven by the step-wise state:
\begin{equation} \label{eq:em_continuous_obs}
    d\tilde{Y}_t = g(X^{EM}_{\eta_h(t)}) \, dt + dB_t, \quad \tilde{Y}_0 = 0.
\end{equation}

Under the reference probability measure $\Pref$ induced by Girsanov's transformation, the observation process (which we denote simply as $Y_t$) is a standard Brownian motion, entirely independent of the state dynamics, and any wide-sense admissible policy $\gamma$ is a measurable mapping on the continuous path space $C([0,T]; \R^M)$.

Let $J_h(\gamma)$ denote the expected cost of the exact sampled diffusion, and $J^{EM, cont}_h(\gamma)$ denote the expected cost of the intermediate continuous-observation Euler--Maruyama model. Under the reference measure $\Pref$, these costs are evaluated using their respective Radon--Nikodym derivatives:
\begin{align}
    G_T^h &:= \exp\left( \int_0^T g(X^h_{t}) \, dY_t - \frac{1}{2} \int_0^T |g(X^h_{t})|^2 \, dt \right), \nonumber\\
    G_T^{EM} &:= \exp\left( \int_0^T g(X^{EM}_{\eta_h(t)}) \, dY_t - \frac{1}{2} \int_0^T |g(X^{EM}_{\eta_h(t)})|^2 \, dt \right), \label{eq:RN_defs}
\end{align}
so that
\begin{align*}
    J_h(\gamma) &= \E_{\Pref}^\gamma\left[ G_T^h \left( \sum_{k=0}^{N_h - 1} \hat{c}_h(X^h_{kh}, U_{kh}) + c_T(X^h_{N_h h}) \right)\right], \\
    J^{EM,cont}_h(\gamma) &= \E_{\Pref}^\gamma\left[ G_T^{EM} \left( \sum_{k=0}^{N_h - 1} \hat{c}_h(X^{EM}_{kh}, U_{kh}) + c_T(X^{EM}_{N_h h}) \right)\right].
\end{align*}
Note that the exponent of $G_T^h$ involves the true state path $X^h_t$ (consistently with the observation model in \eqref{disc_policy}), whereas the exponent of $G_T^{EM}$ involves the piecewise-constant state $X^{EM}_{\eta_h(t)}$ (consistently with \eqref{eq:em_continuous_obs}).

\begin{lemma} \label{lemma:intermediate_convergence}
Suppose Assumptions \ref{cont_bdd} and \ref{assump:lipschitz} hold, the stage cost $c$ and terminal cost $c_T$ are Lipschitz continuous (with constants $L_c, L_{c_T}$) and uniformly bounded, and the observation function $g$ is bounded and Lipschitz continuous with constant $L_g$. Then there exists a constant $C^* < \infty$, depending only on $(K, L_c, L_{c_T}, L_g, \|g\|_\infty, T)$, such that for all $h \in (0, 1]$,
\begin{equation}
    \sup_{\gamma \in \mathfrak{U}_{\mathrm{WS}}} \left| J_h(\gamma) - J^{EM, cont}_h(\gamma) \right| \le C^* \sqrt{h};
\end{equation}
in particular, the costs converge uniformly over all wide-sense admissible policies as $h \to 0$.
\end{lemma}

\begin{proof}
Throughout, all expectations are under $\Pref$ and all constants are independent of the policy $\gamma$. Write
\[
\mathcal{C}^h := \sum_{k=0}^{N_h - 1} \hat{c}_h(X^h_{kh}, U_{kh}) + c_T(X^h_{N_h h}), \qquad
\mathcal{C}^{EM} := \sum_{k=0}^{N_h - 1} \hat{c}_h(X^{EM}_{kh}, U_{kh}) + c_T(X^{EM}_{N_h h}),
\]
and note that $0 \le \mathcal{C}^h, \mathcal{C}^{EM} \le M := T\|c\|_\infty + \|c_T\|_\infty$. By adding and subtracting the cross term $\E_{\Pref}[G_T^{h}\, \mathcal{C}^{EM}]$, the absolute difference is strictly bounded by the sum of two expectations:
\begin{align}
    |J_h(\gamma) - J^{EM, cont}_h(\gamma)| \leq \E_{\Pref}\!\left[ G_T^h \,\big| \mathcal{C}^h - \mathcal{C}^{EM} \big| \right]
    + \E_{\Pref}\!\left[ \left| G_T^h - G_T^{EM} \right| \mathcal{C}^{EM} \right] =: (I) + (II). \label{eq:two_terms}
\end{align}

We first show that the discrete-time stage cost $\hat{c}_h$ is Lipschitz continuous in the state, with a constant of order $O(h)$. Recall the definition \eqref{def:new_cost_fn} of $\hat{c}_h$, and write $X_s^x$ for the solution of the diffusion at time $s$, initialized at $X_0 = x$ and under the constant action $u$. For any two initial states $x$ and $y$, couple the two solutions by driving them with the same Wiener process (synchronous coupling); then
\begin{align*}
|\hat{c}_h(x, u) - \hat{c}_h(y, u)| &\le \E \left[ \int_0^h \left| c(X_s^x, u) - c(X_s^y, u) \right| \, ds \right] \le L_c \int_0^h \E \left[ | X_s^x - X_s^y | \right] \, ds.
\end{align*}
Because $b$ and $\sigma$ are Lipschitz (Assumption \ref{assump:lipschitz}), the solution map is stable with respect to the initial condition: letting $Z_s := X_s^x - X_s^y$, applying the same Cauchy--Schwarz and BDG arguments as in the proof of Theorem~\ref{thm:strong_convergence_mismatch} to
$Z_s = (x-y) + \int_0^s (b(X_r^x)-b(X_r^y))\,dr + \int_0^s (\sigma(X_r^x)-\sigma(X_r^y))\,dW_r$
gives $\E[|Z_s|^2]\le 3|x-y|^2 + C_0\int_0^s \E[|Z_r|^2]\,dr$ with $C_0 := 3(h+4)K^2$; Gr\"onwall's lemma then yields $\E[|Z_s|^2]\le 3 e^{C_0 s}|x-y|^2$, and by Jensen's inequality
\[
\E \left[ | X_s^x - X_s^y | \right] \le \sqrt{3}\, e^{C_0 s / 2} |x - y|.
\]
Substituting this bound into the integral yields
\begin{align*}
|\hat{c}_h(x, u) - \hat{c}_h(y, u)| \le \sqrt{3}\, L_c\, e^{C_0 h/2}\, h\, |x - y| =: \kappa_0\, h\, |x-y|, \qquad \forall u \in \U,
\end{align*}
uniformly in $u$.

We can now bound the first expectation $(I)$. By the above and the Lipschitz continuity of $c_T$,
\[
\big| \mathcal{C}^h - \mathcal{C}^{EM} \big|
\le \kappa_0 h \sum_{k=0}^{N_h-1} |X^h_{kh} - X^{EM}_{kh}| + L_{c_T} |X^h_{N_h h} - X^{EM}_{N_h h}|
\le \left( \kappa_0 T + L_{c_T} \right) S,
\]
where $S := \sup_{0\le t \le T} |X^h_t - X^{EM}_t|$. Note that $G_T^h$ is \emph{not} uniformly bounded; however, by the Cauchy--Schwarz inequality and the second moment bound \eqref{eq:second_moment},
\[
(I) \le \left( \E\left[ (G^h_T)^2 \right] \right)^{1/2} \left( \E\left[ \big|\mathcal{C}^h - \mathcal{C}^{EM}\big|^2 \right] \right)^{1/2}
\le e^{\|g\|_\infty^2 T / 2} \left( \kappa_0 T + L_{c_T} \right) \left( \E[S^2] \right)^{1/2}.
\]
By Theorem~\ref{thm:strong_convergence_mismatch}, $\E[S^2] \le C_2 T e^{2 C_1 T}$ uniformly over policies, and $C_2 = 4C_1K^2(h^2+h) \le 8 C_1 K^2 h$ for $h \le 1$. Hence $(I) \le \kappa_1 \sqrt{h}$ with $\kappa_1 := e^{\|g\|_\infty^2 T / 2} \left( \kappa_0 T + L_{c_T} \right) \left( 8 C_1 K^2 T e^{2 C_1 T} \right)^{1/2}$, independent of $\gamma$.

It remains to bound the second expectation $(II)$. Since $\mathcal{C}^{EM} \le M$, it suffices to bound $\E_{\Pref}|G_T^h - G_T^{EM}|$. Write $G_T^h = e^{L_T^h}$ and $G_T^{EM}=e^{L_T^{EM}}$, where
\[
L_T^h := \int_0^T g(X_t^h)\,dY_t - \tfrac{1}{2}\int_0^T |g(X_t^h)|^2\,dt, \qquad L_T^{EM} := \int_0^T g(X^{EM}_{\eta_h(t)})\,dY_t - \tfrac{1}{2}\int_0^T |g(X^{EM}_{\eta_h(t)})|^2\,dt.
\]
The elementary inequality $|e^b - e^a| = |\int_a^b e^x \, dx| \le |b-a| \left( e^a + e^b \right)$, together with Cauchy--Schwarz and \eqref{eq:second_moment}, gives
\begin{align}
\E\left|G_T^h - G_T^{EM}\right|
&\le \left( \E\left[ (G_T^h + G_T^{EM})^2 \right] \right)^{1/2} \left( \E\left[ |L_T^h - L_T^{EM}|^2 \right] \right)^{1/2} \nonumber\\
&\le 2 e^{\|g\|_\infty^2 T/2} \left( \E\left[ |L_T^h - L_T^{EM}|^2 \right] \right)^{1/2}. \label{eq:G_diff_bound}
\end{align}
It remains to show $\E |L_T^h - L_T^{EM}|^2 = O(h)$. Let $\Delta_t := X^h_t - X^{EM}_{\eta_h(t)}$ and note, by Theorem~\ref{thm:strong_convergence_mismatch} and the single-step increment bound from its proof,
\begin{align}
\int_0^T \E|\Delta_t|^2 \, dt &\le 2 \int_0^T \left( \E|X^h_t - X^{EM}_t|^2 + \E|X^{EM}_t - X^{EM}_{\eta_h(t)}|^2 \right) dt \nonumber\\
&\le 2T\left( C_2 T e^{2C_1 T} + 2K^2(h^2+h) \right) \le \kappa_2 h, \label{eq:Delta_bound}
\end{align}
where $\kappa_2 := 2T\left( 8 C_1 K^2 T e^{2C_1 T} + 4K^2 \right)$. Using $|a+b|^2 \le 2|a|^2 + 2|b|^2$, It\^o's isometry (recall $Y$ is a standard Brownian motion under $\Pref$), and the Lipschitz continuity of $g$,
\begin{align*}
\E\left| \int_0^T \left( g(X_t^h)-g(X^{EM}_{\eta_h(t)}) \right) dY_t \right|^2 &= \int_0^T \E\left| g(X_t^h)-g(X^{EM}_{\eta_h(t)}) \right|^2 dt \le L_g^2 \int_0^T \E|\Delta_t|^2\,dt \le L_g^2 \kappa_2 h,
\end{align*}
while for the bounded-variation part, using $\big| |g(x)|^2 - |g(y)|^2 \big| \le 2\|g\|_\infty L_g |x-y|$ and Cauchy--Schwarz in time,
\begin{align*}
\E\left| \int_0^T \left( |g(X_t^h)|^2 - |g(X^{EM}_{\eta_h(t)})|^2 \right) dt \right|^2
\le T \int_0^T 4 \|g\|_\infty^2 L_g^2\, \E |\Delta_t|^2 \, dt \le 4 T \|g\|_\infty^2 L_g^2 \kappa_2 h.
\end{align*}
Combining the two bounds,
\[
\E|L_T^h - L_T^{EM}|^2 \le 2 L_g^2 \kappa_2 h + 2 T \|g\|_\infty^2 L_g^2 \kappa_2 h =: \kappa_3 h,
\]
and by \eqref{eq:G_diff_bound}, $(II) \le 2 M e^{\|g\|_\infty^2 T/2} \sqrt{\kappa_3 h}$.

Collecting the bounds on $(I)$ and $(II)$, which hold uniformly over $\gamma \in \mathfrak{U}_{\mathrm{WS}}$, completes the proof.
\end{proof}

To formulate a computable finite POMDP, the controller cannot have access to the continuous path $Y_t$; instead, the policy must depend strictly on the discrete observation increments $\bar{Y}_{k+1}^h = Y_{(k+1)h} - Y_{kh}$. We now prove that restricting the policy space to these discrete samples incurs no asymptotic loss, which is the main result of this section.

\begin{theorem}\label{thm:main_approximation}
Let $\gamma^{EM}_h$ be the optimal policy derived from the fully discrete Euler--Maruyama POMDP (which uses only the discrete observation samples); such an optimal policy exists by the weak Feller property established in Section~\ref{sec:cont_reg} and the compactness of $\U$. Under the conditions of Lemma~\ref{lemma:intermediate_convergence}, $\gamma^{EM}_h$ is asymptotically near optimal for the exact sampled diffusion model:
\begin{equation}\label{eq:main_sampled}
    J_h(\gamma^{EM}_h) \leq \inf_{\gamma \in \mathfrak{U}_{\mathrm{WS}}} J_h(\gamma) + \epsilon(h),
\end{equation}
where $\epsilon(h) \to 0$ as $h \to 0$; consequently, it is also asymptotically near optimal for the original continuous-time partially observed diffusion \eqref{cont_model}--\eqref{obs_model}:
\begin{equation}\label{eq:direct_ct_optimality}
J_T(x, \gamma^{EM}_h) \le J_T^*(x) + \tilde{\epsilon}(h), \qquad \tilde{\epsilon}(h) \to 0 \text{ as } h \to 0,
\end{equation}
where $J_T^*(x) = \inf_{\gamma \in \mathfrak{U}_{\mathrm{WS}}} J_T(x,\gamma)$. More generally, any near optimal policy for the Euler--Maruyama model is near optimal for the original continuous-time partially observed diffusion as $h \to 0$.
\end{theorem}

\begin{proof}
Let $\Gamma_{path}$ denote the space of all admissible path-dependent (relaxed) policies, described by measurable stage maps $\gamma_k : C([0,kh]; \R^M) \to \cP(\U)$ as in Section~\ref{sec:cont_existence_disc}, and let $\Gamma_{disc} \subset \Gamma_{path}$ denote the restricted space of policies that depend solely on the discrete sample sequence $\{Y_{kh}\}_{k=0}^{N_h}$.

Under the reference measure $\Pref$, the observation $Y_t$ is a standard Brownian motion. For any measurable policy $\gamma \in \Gamma_{path}$ and any $\delta > 0$, Lusin's theorem and the Dugundji extension, applied to each stage map exactly as in Section~\ref{sec:cont_existence_disc}, provide a policy $\gamma_C$ with continuous stage maps and an event $\Omega_\delta$ with $\Pref(\Omega_\delta) > 1 - \delta$ such that $\gamma_C = \gamma$ on $\Omega_\delta$. Because the two policies differ only on the complement $\Omega_\delta^c$, and the total cost is bounded by $M$, the performance difference under the intermediate model is strictly bounded, exactly as in Section~\ref{sec:cont_existence_disc}:
\begin{equation}
    |J^{EM, cont}_h(\gamma) - J^{EM, cont}_h(\gamma_C)| \leq M\, \E_{\Pref}\!\left[ \left( G_T^{EM,\gamma} + G_T^{EM,\gamma_C} \right) \mathbf{1}_{\Omega_\delta^c} \right] \le 2 M e^{\|g\|_\infty^2 T/2} \sqrt{\delta},
\end{equation}
where $G_T^{EM,\gamma}$ and $G_T^{EM,\gamma_C}$ denote the Radon--Nikodym derivatives under the two policies, and the last inequality follows from the Cauchy--Schwarz inequality and the second moment bound \eqref{eq:second_moment}. Hence, as $\delta \to 0$, any measurable policy can be approximated arbitrarily well by a policy with continuous stage maps.

Next, we evaluate this continuous policy $\gamma_C$ on the discrete measurements. Define the piecewise-constant interpolation of the sampled path as $\hat{Y}^h_t = Y_{kh}$ for $t \in [kh, (k+1)h)$. Because the sample paths of Brownian motion are almost surely uniformly continuous on $[0,T]$, the interpolated path converges uniformly to the true path: $\sup_{0 \leq t \leq T} |\hat{Y}^h_t - Y_t| \to 0$ almost surely as $h \to 0$. Because the stage maps of $\gamma_C$ are continuous with respect to the uniform topology on the path space, evaluating them on the discretized path yields $\gamma_C(\hat{Y}^h) \to \gamma_C(Y)$ almost surely, and $\gamma_C(\hat{Y}^h)$ depends entirely on the finite set of samples $\{Y_{kh}\}_{k=0}^{N_h}$; that is, $\gamma_C(\hat{Y}^h) \in \Gamma_{disc}$. As in Section~\ref{sec:cont_existence_disc}, almost sure convergence of the actions implies weak convergence of the joint laws in $\mathfrak{U}_{\mathrm{WS}}$, and the corresponding costs converge by the continuity of the cost in the policy (Theorem~\ref{thm:PY_existence}(i)) together with Lemma~\ref{lemma:intermediate_convergence}.

Therefore, the cost infimum over the restricted discrete policy space is asymptotically equivalent to the infimum over the full policy space:
\begin{equation}\label{eq:eps2}
    \inf_{\gamma \in \Gamma_{disc}} J^{EM, cont}_h(\gamma) \leq \inf_{\gamma \in \mathfrak{U}_{\mathrm{WS}}} J^{EM, cont}_h(\gamma) + \epsilon_2(h),
\end{equation}
where $\epsilon_2(h) \to 0$ as $h \to 0$.

Now let $\epsilon_1(h) := C^* \sqrt{h}$ be the uniform bound of Lemma~\ref{lemma:intermediate_convergence}, and recall that $J^{EM,cont}_h = J^{EM}_h$ on $\Gamma_{disc}$, so that the optimality of $\gamma^{EM}_h$ for the fully discrete model gives $J^{EM,cont}_h(\gamma^{EM}_h) = \inf_{\gamma' \in \Gamma_{disc}} J^{EM,cont}_h(\gamma')$. Then, for any $\gamma \in \mathfrak{U}_{\mathrm{WS}}$:
\begin{align*}
    J_h(\gamma^{EM}_h) &\leq J^{EM, cont}_h(\gamma^{EM}_h) + \epsilon_1(h) \\
    &= \inf_{\gamma' \in \Gamma_{disc}} J^{EM, cont}_h(\gamma') + \epsilon_1(h) \\
    &\leq J^{EM, cont}_h(\gamma) + \epsilon_1(h) + \epsilon_2(h) \\
    &\leq J_h(\gamma) + 2\epsilon_1(h) + \epsilon_2(h).
\end{align*}
Defining $\epsilon(h) := 2\epsilon_1(h) + \epsilon_2(h)$ and taking the infimum over all policies $\gamma \in \mathfrak{U}_{\mathrm{WS}}$ on the right-hand side yields \eqref{eq:main_sampled}. Thus, optimizing the computable discrete-time Euler--Maruyama POMDP directly yields a policy that performs within $\epsilon(h)$ of the best possible policy on the exact sampled continuous-time diffusion.

Finally, since the control induced by $\gamma^{EM}_h$ is piecewise constant on $h$-interval, we have $J_h(\gamma^{EM}_h) = J_T(x, \gamma^{EM}_h)$ (Section~\ref{sec:cont_existence_disc}), while $\inf_{\gamma} J_h(\gamma) \to J_T^*(x)$ as $h \to 0$ by Theorem~\ref{PY_pathvalued}(i); hence \eqref{eq:direct_ct_optimality} holds with $\tilde{\epsilon}(h) := \epsilon(h) + \left( \inf_{\gamma} J_h(\gamma) - J_T^*(x) \right)$. For the last claim of the theorem, if $\gamma_h$ is instead an $\varepsilon'$-optimal policy for the Euler--Maruyama model, the second line of the chain above holds with an additional $\varepsilon'$, and the same argument gives \eqref{eq:main_sampled} and \eqref{eq:direct_ct_optimality} with $\epsilon(h) + \varepsilon'$ and $\tilde{\epsilon}(h) + \varepsilon'$, respectively.
\end{proof}

So far we have considered the finite-horizon criterion. The same argument also applies to the infinite-horizon discounted cost criterion, which we now describe. For the original continuous-time controlled diffusion, the discounted cost under a policy $\gamma$ is defined as
\[
J_\alpha(x, \gamma) := \E_x^\gamma \left[ \int_0^\infty e^{-\alpha t} c(X_t, U_t) \, dt \right],
\]
for some discount rate $\alpha > 0$, with optimal cost $J_\alpha^*(x) := \inf_{\gamma \in \mathfrak{U}_{\mathrm{WS}}} J_\alpha(x, \gamma)$. By applying a piecewise-constant policy over intervals of length $h$, we arrive at the equivalent discrete-time discounted cost
\[
J_{\alpha, h}(x, \gamma) := \E_x^\gamma \left[ \sum_{k=0}^{\infty} \beta^k\, \hat{c}^{\,\alpha}_h(X_{kh}^h, U_{kh}) \right], \qquad \beta := e^{-\alpha h} \in (0,1),
\]
where the discrete-time stage cost is
\[
\hat{c}^{\,\alpha}_h(x, u) := \E \left[ \int_0^h e^{-\alpha t} c(X_t, u) \, dt \,\bigg|\, X_0 = x \right].
\]
The cost of the fully discrete Euler--Maruyama POMDP under this criterion, denoted $J^{EM}_{\alpha,h}(\gamma)$, is defined analogously, with $X^{EM}$ in place of $X^h$.

\begin{theorem}\label{thm:main_approximation_discounted}
Suppose the conditions of Lemma~\ref{lemma:intermediate_convergence} hold. For $h > 0$, let $\gamma^{EM}_{\alpha, h}$ be an optimal policy of the fully discrete Euler--Maruyama POMDP under the discounted cost criterion. Then $\gamma^{EM}_{\alpha, h}$ is asymptotically near optimal for the original continuous-time problem:
\[
J_\alpha(x, \gamma^{EM}_{\alpha,h}) \le J_\alpha^*(x) + \epsilon(h),
\]
where $\epsilon(h) \to 0$ as $h \to 0$.
\end{theorem}

\begin{remark}
To the best of our knowledge, for the average cost criterion, obtaining near optimal policies via time discretization is still an open problem; for this reason, we do not generalize our results to the average cost criterion.
\end{remark}

\section{Regularity Properties of the POMDP obtained via Euler-Maruyama Approximation}\label{asymptotic_continous}

Now that we have a discrete-time POMDP via the Euler--Maruyama approximation \eqref{eq:euler_state}--\eqref{eq:euler_obs}, we obtain regularity properties and their implications on finite model approximations.  For this section we prove that filter kernel for the new approximate model is weak Feller, so we can quantize the belief state to find the near optimal policy.

\subsection{Discrete Time Partially Observable Markov Decision Process}
A discrete-time Partially Observable Markov Decision Process (POMDP) is characterized by the tuple $(\mathbb{X}, \mathbb{U}, \mathbb{Y}, \mathcal{T}, \mathcal{Q}, c)$. The state process takes values in a Polish metric space $(\mathbb{X}, d)$, and the measurement sequence $\{Y_k\}$ takes values in a standard Borel space $\mathbb{Y}$. The initial state $X_0$ distributed according to a probability measure $\mu \in \mathcal{P}(\mathbb{X})$. We denote by $\mathbb{B}(\mathbb{X})$ the Borel $\sigma$-field on $\mathbb{X}$, and by $\mathcal{P}(\mathbb{X})$ the space of probability measures on $(\mathbb{X}, \mathbb{B}(\mathbb{X}))$, equipped with the topology of weak convergence. The set of continuous and bounded functions on $\mathbb{X}$ is denoted by $\mathbb{C}(\mathbb{X})$.  The signals and observations are governed by the state transition kernel $\mathcal{T}$ and the measurement kernel $\mathcal{Q}$, respectively. For Borel sets $B_x \subset \mathbb{X}$ and $B_y \subset \mathbb{Y}$:
\begin{align}
\mathcal{T}(B_x \mid x, u) &:= P(X_{t+1} \in B_x \mid X_t=x, U_t=u), \\
\mathcal{Q}(B_y \mid x, u) &:= P(Y_{t} \in B_y \mid X_{t}=x).
\end{align}

At each time step $t$, the decision maker selects a control action $U_t$ from a Polish space $\mathbb{U}$, incurring a stage cost $c(X_t, U_t)$. The decision maker only has access to the measurement sequence and past control actions. An admissible policy $\gamma$ is a sequence of control functions $\{\gamma_t\}_{t \in \mathbb{Z}_{+}}$, where each $\gamma_t$ is measurable with respect to the $\sigma$-algebra generated by the available information:
\begin{equation}
I_t = \{Y_{[0, t]}, U_{[0, t-1]}\}, \quad \text{with} \quad I_0 = \{Y_0\},
\end{equation}
so that $U_t = \gamma_k(I_t)$. We denote by $\Gamma$ the set of all such admissible policies.
Given the initial distribution $\mu$ and a policy $\gamma \in \Gamma$, the Ionescu-Tulcea theorem \cite{HernandezLermaMCP} guarantees the existence of a unique probability measure, denoted $P_\mu^\gamma$, on the infinite product space $(\Omega, \mathcal{F})$ consistent with the finite-dimensional distributions.

We are interested in minimizing the finite-horizon total cost criterion over $T$ stages. Given a continuous and bounded stage cost $c: \mathbb{X} \times \mathbb{U} \to \mathbb{R}_{+}$ and a continuous terminal cost $c_T: \mathbb{X} \to \mathbb{R}_{+}$, the expected cost is:
\begin{equation}
J_T(c, \mu, \mathcal{T}, \mathcal{Q}, \gamma) := \mathbb{E}_\mu^\gamma \left[ \sum_{k=0}^{T-1} c(X_k, U_k) + c_T(X_T) \right].
\end{equation}
The optimal finite-horizon cost is defined as $J_T^*(c, \mu, \mathcal{T}, \mathcal{Q}) := \inf_{\gamma \in \Gamma} J_T(c, \mu, \mathcal{T}, \mathcal{Q}, \gamma)$.

\subsection{Belief-MDP Reduction}

A POMDP can be reduced to a completely observable MDP where the state is the posterior distribution of the system, known as the belief state. The belief state at time $t$ is defined as:
\begin{equation}
\pi_t(\cdot) := P \{X_t \in \cdot \mid Y_0, \ldots, Y_t, U_0, \ldots, U_{t-1}\} \in \mathcal{P}(\mathbb{X}).
\end{equation}

This equivalent model is the belief-MDP. It has the state space $\mathcal{P}(\mathbb{X})$ and action space $\mathbb{U}$. Since $\mathbb{X}$ is a Borel space, the space $\mathcal{P}(\mathbb{X})$ equipped with the topology of weak convergence is metrizable by the Prokhorov metric, which makes $\mathcal{P}(\mathbb{X})$ itself a Borel space. 

The recursive update of the belief state is governed by a measurable non-linear filtering map $F(\pi, u, y): \mathcal{P}(\mathbb{X}) \times \mathbb{U} \times \mathbb{Y} \to \mathcal{P}(\mathbb{X})$, defined via Bayes' rule:
\begin{equation}
F(\pi, u, y)(dx') := P \{X_{t+1} \in dx' \mid \pi_t=\pi, U_t=u, Y_{t+1}=y\}.
\end{equation}

Let $P(\cdot \mid \pi, u) := P\{Y_{k+1} \in \cdot \mid \pi_k=\pi, U_k=u\}$ denote the stochastic kernel on $\mathbb{Y}$ given $\mathcal{P}(\mathbb{X}) \times \mathbb{U}$. The transition probability $\eta$ of the belief-MDP is constructed by:
\begin{equation}
\eta(\cdot \mid \pi, u) = \int_{\mathbb{Y}} \mathbf{1}_{\{F(\pi, u, y) \in \cdot\}} P(dy \mid \pi, u).
\end{equation}

The one-stage cost function $\tilde{c}$ of the belief-MDP is calculated by taking expectation of cost over the current belief state:
\begin{equation}
\tilde{c}(\pi, u) := \int_{\mathbb{X}} c(x, u) \pi(dx).
\end{equation}
Because the original cost function $c(x, u)$ is continuous and bounded on $\mathbb{X} \times \mathbb{U}$, by the dominated convergence theorem $\tilde{c}(\pi, u)$ is also continuous and bounded.

The belief-MDP is defined by the tuple $(\mathcal{P}(\mathbb{X}), \mathbb{U}, \eta, \tilde{c})$. For finite-horizon problems, it is a well-established result that the belief $\pi_t$ is a sufficient statistic, meaning that without losing optimality the policies can be restricted to Markov policies on the belief state  (see \cite{Yus76,Rhe74,Blackwell2}).

\subsection{Weak Feller Continuity of the Belief-MDP}

For quantization of the belief space and to ensure the continuity of the finite-horizon value function, we must establish the weak Feller property of the filter kernel $\eta$ defined above. Specifically, $\eta$ is weak Feller if for every $f \in \mathbb{C}(\mathcal{P}(\mathbb{X}))$, the mapping:
\begin{equation}
(\pi, u) \mapsto \int_{\mathcal{P}(\mathbb{X})} f(\pi') \eta(d\pi' \mid \pi, u)
\end{equation}
is continuous in $(\pi, u)$. 

We now present two sets of conditions under which the filter kernel satisfies the weak Feller property.
\begin{assumption}\label{assum:weak_feller_1}
\leavevmode
\begin{enumerate}
    \item[(i)] The transition probability $\mathcal{T}(\cdot \mid x, u)$ is weakly continuous in $(x, u)$. That is, for any sequence $(x_n, u_n) \to (x, u)$, $\mathcal{T}(\cdot \mid x_n, u_n) \to \mathcal{T}(\cdot \mid x, u)$ weakly.
    \item[(ii)] The observation channel $\mathcal{Q}(\cdot \mid x, u)$ is continuous in total variation. That is, for any sequence $(x_n, u_n) \to (x, u)$, $\mathcal{Q}(\cdot \mid x_n, u_n) \to \mathcal{Q}(\cdot \mid x, u)$ in total variation.
\end{enumerate}
\end{assumption}

\begin{remark}
    Note that this assumption is for a more general setup in which the observation channel may also depend on the control action.
\end{remark}

\begin{assumption}\label{assum:weak_feller_2}
\leavevmode
\begin{enumerate}
    \item[(i)] The transition probability $\mathcal{T}(\cdot \mid x, u)$ is continuous in total variation in $(x, u)$. That is, for any $(x_n, u_n) \to (x, u)$, $\mathcal{T}(\cdot \mid x_n, u_n) \to \mathcal{T}(\cdot \mid x, u)$ in total variation.
    \item[(ii)] The observation channel $\mathcal{Q}(\cdot \mid x)$ is independent of the control variable.
\end{enumerate}
\end{assumption}

\begin{theorem}\label{thm:weakfeller}
The transition probability $\eta(\cdot \mid \pi, u)$ of the filter process satisfies the weak Feller property under either:
(i)\cite{FeKaZg14} Assumption \ref{assum:weak_feller_1}; or (ii)\cite{KSYWeakFellerSysCont} Assumption \ref{assum:weak_feller_2}.
\end{theorem}

\subsection{Regularity Properties of Euler-Maruyama Model}\label{sec:cont_reg}

Now that we have derived a discrete-time POMDP via the Euler-Maruyama approximation in \eqref{eq:euler_state} and \eqref{eq:euler_obs}, in this section we verify that it satisfies Assumption \ref{assum:weak_feller_1}. In particular, we prove that the transition kernel of this approximate model is weak Feller and that its observation kernel is Lipschitz continuous in total variation. These properties permit quantization of the belief spaces to compute a near-optimal policy for the finite model.

We denote the transition kernel of the state equation \eqref{eq:euler_state} by $\cT_h$ and the observation kernel by $Q_h$; the stage cost $\hat{c}_h$ is given by \eqref{def:new_cost_fn}. Throughout this subsection we write $X^{EM}_k := X^{EM}_{kh}$ and $U_k := U_{kh}$ for brevity. From the Euler--Maruyama state equation, the transition over the interval $[kh, (k+1)h]$ is
\begin{equation}
X_{k+1}^{EM} = X_k^{EM} + b(X_k^{EM}, U_k)h + \sigma(X_k^{EM})\Delta W_k,
\end{equation}
where the Brownian increment satisfies $\Delta W_k \sim \mathcal{N}(0, h I_d)$. Conditioned on the current state $x$ and action $u$, the next state is Gaussian, and the transition kernel is
\begin{equation}\label{eq:T_h_def}
\cT_h(\cdot \mid x, u) = \mathcal{N}\left(x + b(x, u)h, \, h\sigma(x)\sigma(x)^\top\right).
\end{equation}
For the observation process, from \eqref{eq:euler_obs} the measurement is the increment $\bar{Y}_{k+1}^h := Y_{(k+1)h}^{EM} - Y_{kh}^{EM}$, whose update is
\begin{equation}\label{eq:obs_increment}
\bar{Y}_{k+1}^h = g(X_k^{EM})h + \Delta B_k,
\end{equation}
where $\Delta B_k \sim \mathcal{N}(0, h I_M)$ is independent of the state; hence
\begin{equation}\label{eq:Q_h_def}
Q_h(\cdot \mid x) = \mathcal{N}\left(g(x)h, \, h I_M\right).
\end{equation}

% \begin{theorem}\label{thm:w1_continuity}
% Suppose Assumption \ref{assump:lipschitz} holds. Then, the transition kernel $\mathcal{T}_h$ is Lipschitz continuous in the initial state with respect to the $W_1$ metric:
% \[
% W_1\left(\mathcal{T}_h(\cdot \mid x, u), \mathcal{T}_h(\cdot \mid x', u)\right) \le \left(1 + K h + K \sqrt{h}\right) d(x, x')
% \]
% \end{theorem}

% \begin{proof}
% Consider two initial states $x, x' \in \mathbb{R}^d$ under the same control action $u$. 
% We couple the noise variables by using the same realization of \( \Delta W_k \):
% \begin{align*}
% X_{k+1}^h(x,u) &= x + b(x,u)h + \sigma(x)\Delta W_k \\
% X_{k+1}^h(x',u) &= x' + b(x',u)h + \sigma(x')\Delta W_k
% \end{align*}
% Subtracting the equations and taking the expectation of the norm yields:
% \[
% \mathbb{E}\left[\|X_{k+1}^h(x,u) - X_{k+1}^h(x',u)\|\right] \le d(x, x')  + h \|b(x,u) - b(x',u)\| + \mathbb{E}\left[\|(\sigma(x) - \sigma(x'))\Delta W_k\|\right]
% \]
% Using the Lipschitz property of $b$ and $\sigma$, and noting that $\mathbb{E}[\|\Delta W_k\|] \le \sqrt{hd}$:
% \[
% \mathbb{E}\left[\|X_{k+1}^h(x,u) - X_{k+1}^h(x',u)\|\right] \le d(x, x') + h K d(x, x')  + \sqrt{hd} K d(x, x') 
% \]
% By the definition of the $W_1$ metric as the infimum over all couplings, we have:

% \[
% W_1\left(\mathcal{T}_h(\cdot \mid x, u), \mathcal{T}_h(\cdot \mid x', u)\right) \le \left(1 + K h + K \sqrt{hd}\right) d(x, x') 
% \]

% \end{proof}

\begin{theorem}\label{thm:joint_weak_continuity}
Suppose Assumptions \ref{cont_bdd}, \ref{nondeg} and \ref{assump:lipschitz} hold. The transition kernel $\cT_h(\cdot \mid x, u)$ is weakly continuous in $(x, u)$; in fact, for any sequence $(x_n, u_n) \to (x, u)$, the kernel converges in the Wasserstein-1 metric:
\[
\lim_{n \to \infty} W_1\left(\cT_h(\cdot \mid x_n, u_n), \cT_h(\cdot \mid x, u)\right) = 0.
\]
\end{theorem}

\begin{proof}
By Assumption \ref{cont_bdd}, the drift $b(x, u)$ is jointly continuous in $(x, u)$ and the diffusion matrix $\sigma(x)$ is continuous. Let $(x_n, u_n) \to (x, u)$. As in the previous proof, we construct a coupling using the identical realization of the Brownian increment $\Delta W_k \sim \mathcal{N}(0, h I_d)$ for both transitions:
\begin{align*}
X_{k+1}^{EM}(x_n, u_n) &= x_n + b(x_n, u_n)h + \sigma(x_n)\Delta W_k, \\
X_{k+1}^{EM}(x, u) &= x + b(x, u)h + \sigma(x)\Delta W_k.
\end{align*}
Taking expectations and using $\E[|\Delta W_k|] \le \sqrt{hd}$,
\begin{align*}
\E\!\left[|X_{k+1}^{EM}(x_n, u_n) - X_{k+1}^{EM}(x, u)|\right] &\le |x_n - x| + h |b(x_n, u_n) - b(x, u)| + \sqrt{hd}\, \|\sigma(x_n) - \sigma(x)\|.
\end{align*}
By the definition of the $W_1$ metric as an infimum over couplings,
\[
W_1\left(\cT_h(\cdot \mid x_n, u_n), \cT_h(\cdot \mid x, u)\right) \le |x_n - x| + h |b(x_n, u_n) - b(x, u)| + \sqrt{hd}\, \|\sigma(x_n) - \sigma(x)\|.
\]
As $n \to \infty$, we have $|x_n - x| \to 0$; by the joint continuity of $b$, it follows that $|b(x_n, u_n) - b(x, u)| \to 0$, and similarly the continuity of $\sigma$ guarantees that $\|\sigma(x_n) - \sigma(x)\| \to 0$. Consequently, the upper bound vanishes. Because convergence in the Wasserstein-1 metric implies weak convergence of probability measures, we conclude that $\cT_h(\cdot \mid x_n, u_n) \to \cT_h(\cdot \mid x, u)$ weakly.
\end{proof}

\begin{theorem}\label{thm:tv_continuity_cont}
If $g$ is continuous, then the observation kernel $Q_h$ is continuous in the state with respect to the total variation distance:
\[
\|Q_h(\cdot \mid x_n) - Q_h(\cdot \mid x)\|_{TV} \to 0 \qquad \text{whenever } x_n \to x.
\]
\end{theorem}

\begin{proof}
The kernels $Q_h(\cdot \mid x_n) = \mathcal{N}(g(x_n)h, hI_M)$ and $Q_h(\cdot \mid x) = \mathcal{N}(g(x)h, hI_M)$ are Gaussians with identical covariance, so by Pinsker's inequality  \cite[Lemma~2.5]{tsybakov2009}, together with the Gaussian relative entropy formula (see, e.g., \cite[Chapter~9, Section~1]{Kullback59})
\[
D:=D\left( \mathcal{N}(\mu, \Sigma) \,\big\|\, \mathcal{N}(\mu', \Sigma) \right) = \tfrac{1}{2} (\mu - \mu')^\top \Sigma^{-1} (\mu - \mu'),
\]
gives $\|\mathcal{N}(\mu, \Sigma) - \mathcal{N}(\mu', \Sigma)\|_{TV} \le \sqrt{2D}$.
If $x_n \to x$, $D \to 0$ by the continuity of $g$, and the claim follows.
\end{proof}
The statement also follows by Scheffe's lemma. Each kernel admits a density with respect to the Lebesgue measure on $\R^M$,
\[
q_h(x,y) = (2\pi h)^{-M/2} \exp\Bigl( -\tfrac{1}{2h}\,\bigl|y - g(x)h\bigr|^2 \Bigr),
\]
and if $x_n \to x$ then $q_h(x_n, y) \to q_h(x,y)$ for every $y \in \R^M$, by the continuity of $g$. Since $q_h(x_n, \cdot)$ and $q_h(x,\cdot)$ are probability densities, by Scheffe's lemma implies convergence in $L^1(\R^M)$, which is exactly $\|Q_h(\cdot \mid x_n) - Q_h(\cdot \mid x)\|_{TV} \to 0$.

As a result of Theorems \ref{thm:joint_weak_continuity} and \ref{thm:tv_continuity_cont}, we obtain the following.

\begin{corollary}\label{cor:filter_weak_feller}
Suppose Assumptions \ref{cont_bdd}, \ref{nondeg}, and \ref{assump:lipschitz} hold; recall that $g$ is continuous and bounded. Fix $h > 0$. Because the state transition kernel $\cT_h$ is weakly continuous (Theorem~\ref{thm:joint_weak_continuity}) and the observation kernel $Q_h$ is continuous in total variation (Theorem~\ref{thm:tv_continuity_cont}), by Theorem~\ref{thm:weakfeller}-(i) the transition kernel $\eta_h$ of the non-linear filter satisfies the weak Feller property.
\end{corollary}

Given the weak Feller property of $\eta_h$, we can now apply the standard finite-model approximation results for weakly Feller belief MDPs \cite{HernandezLermaMCP, SaYuLi15c,SYLTAC2017POMDP,SaLiYuSpringer} to construct a finite-state, finite-action approximation of the discrete-time Euler--Maruyama POMDP whose optimal policy is near optimal for the original continuous-time controlled diffusion. The next section makes this explicit.

\section{Asymptotic Optimality By Quantization}

With the regularity properties of Section~\ref{asymptotic_continous}, we now quantize the action, belief and measurement spaces that the policy computed on the quantized model remains near optimal for the original model.

\subsection{Action Quantization}

First, we quantize the action space $\mathbb{U}$. By quantizing the action space, we can obtain a near-optimal policy using a finite set of actions. Since the action space $\mathbb{U}$ is assumed to be compact, it is totally bounded. Let $d_{\mathbb{U}}$ denote the metric on $\mathbb{U}$. There exists a sequence of finite sets $\mathbb{U}_n = \{u_{n, 1}, \ldots, u_{n, k_n}\} \subset \mathbb{U}$ such that $\mathbb{U}_n$ forms a $1/n$-net in $\mathbb{U}$. This means that for each $n$:
$$
\min_{i \in \{1, \ldots, k_n\}} d_{\mathbb{U}}(u, u_{n, i}) < 1/n \quad \text{for all } u \in \mathbb{U}.
$$

We define a mapping $\Upsilon_n$ that projects any policy $\gamma$ onto this finite set of actions:
$$
\Upsilon_n(\gamma)(\pi) = \arg\min_{u \in \mathbb{U}_n} d_{\mathbb{U}}(\gamma(\pi), u),
$$
where ties are broken ensures to preserve measurability.

Because the belief transition kernel satisfies the weak Feller property (Corollary~\ref{cor:filter_weak_feller}) and the cost functions are continuous and bounded, solving the model with this finite action space yields a near optimal solution for the finite-horizon problem as $n \to \infty$; see \cite[Chapter 3]{SaLiYuSpringer}. (Similarly, for the discounted cost criterion, the difference between the optimal cost of the finite-action model and that of the original model converges to $0$ uniformly on every compact subset of $\cP(\X)$ \cite[Lemma 3.19]{SaLiYuSpringer}.) Therefore, for the remainder of the analysis, we can assume without loss of generality that the action space $\U$ is finite.

\subsection{Belief Quantization}

The next step is to discretize the continuous belief space $\cP(\X)$, where the underlying state space is $\X = \R^d$. Throughout this subsection, the time step $h$ of the Euler--Maruyama POMDP is fixed and suppressed from the notation: we write $\eta := \eta_h$ for its belief transition kernel, $\tilde{c}(\pi, u) := \int_\X \hat{c}_h(x,u)\, \pi(dx)$ for the induced stage cost, and $N := N_h$ for the horizon. Under our standing assumptions (bounded, Lipschitz drift and diffusion) and an initial distribution with a finite second moment, the conditional law $\pi_k$ of the state given the observation history takes values, almost surely, in the space of probability measures with finite second moments,
\[
 \cP_2(\X) := \Bigl\{\mu \in \cP(\X) : \int_{\X} |x|^2\, \mu(dx) < \infty \Bigr\}.
\]
Define an increasing family of relatively compact subsets
\[
 M_n := \Bigl\{\mu \in \cP(\X) : \int_{\X} |x|^2\, \mu(dx) < n \Bigr\},\qquad \cP_2(\X) = \bigcup_{n=1}^{\infty} M_n.
\]
For every $\mu \in M_n$ and every $R > 0$, $\mu(\{x : |x| > R\}) \le n/R^2$ by Markov's inequality, so the family $M_n$ is uniformly tight; hence, by Prohorov's theorem, $M_n$ is relatively compact in $\cP(\X)$ under the weak topology. Let $\{\vartheta_n\}$ be a sequence of probability measures with $\vartheta_n \in \cP(M_n^c)$ for each $n\ge 1$, acting as a representative distribution over the complement of $M_n$.

First, we define a sequence of truncated MDPs, denoted $\text{c-MDP}_n$, to approximate the exact belief MDP. For each $n$, the state space is restricted to $M_n \cup \{\Delta_n\}$, where the isolated point $\Delta_n \notin M_n$ represents the complement $M_n^c$. We define the transition probability $\eta_n$ on $M_n \cup \{\Delta_n\}$ and the one-stage cost function $\tilde{c}_n$ on $(M_n \cup \{\Delta_n\}) \times \U$ as follows:
\[
\eta_n(\cdot \mid \pi, u) =
\begin{cases}
\eta(\cdot \cap M_n \mid \pi, u) + \eta(M_n^c \mid \pi, u)\, \delta_{\Delta_n}(\cdot), & \text{if } \pi \in M_n, \\[2pt]
\displaystyle\int_{M_n^c} \left( \eta(\cdot \cap M_n \mid z, u) + \eta(M_n^c \mid z, u)\, \delta_{\Delta_n}(\cdot) \right) \vartheta_n(dz), & \text{if } \pi = \Delta_n,
\end{cases}
\]
and
\[
\tilde{c}_n(\pi, u) =
\begin{cases}
\tilde{c}(\pi, u), & \text{if } \pi \in M_n, \\[2pt]
\displaystyle\int_{M_n^c} \tilde{c}(z, u)\, \vartheta_n(dz), & \text{if } \pi = \Delta_n.
\end{cases}
\]

With these components, $\text{c-MDP}_n = (M_n \cup \{\Delta_n\}, \U, \eta_n, \tilde{c}_n)$ is a well-defined Markov decision process. Since the original cost function $c$ is uniformly bounded, the expected cost accumulated while the belief state is in the overflow bin $\Delta_n$ is asymptotically negligible. Therefore, as $n \to \infty$, the optimal cost of this truncated model converges to the optimal cost of the original model uniformly on any compact subset of initial beliefs in $\cP_2(\X)$; see \cite[Chapter 4]{SaLiYuSpringer}.

Having reduced the problem to $\text{c-MDP}_n$, we now quantize the relatively compact set $M_n$ to compute a near optimal policy. We partition $M_n$ into a finite collection of disjoint Borel sets $\{B_{n,i}\}_{i=1}^{N_n-1}$ and treat the last cell as the overflow bin $B_{n,N_n} = \{\Delta_n\}$. The resolution of this quantization is measured by the maximum cell diameter under the Wasserstein-1 metric:
\begin{equation}
L_n^{-} := \max_{i=1, \ldots, N_n-1} \sup_{\pi, \pi^{\prime} \in B_{n,i}} W_1(\pi, \pi^{\prime}).
\end{equation}

The asymptotic convergence of this quantization follows from standard discretization results for Borel MDPs (see \cite{Ber75, SaLiYuSpringer}). Because the belief transition kernel $\eta$ is weak Feller and the cost functions are continuous and bounded, the discretization error remains bounded over the finite horizon $N$ and vanishes with the resolution. Specifically, for any compact subset $\cK \subset \X$, by the finite-state quantization convergence results for weak Feller belief MDPs (see \cite[Theorem~4.27]{SaLiYuSpringer} and \cite[Theorem~3.16]{SYLTAC2017POMDP}), the optimal value function of the fully quantized finite-state model, $\hat{J}_N$, converges uniformly to the optimal value function of the exact belief MDP, $J_N^*$:
\begin{equation} \label{eq:value_convergence}
\sup_{\pi_0 \in \cP(\cK)} \left|\hat{J}_N(\pi_0) - J_N^*(\pi_0)\right| \to 0 \quad \text{as } L_n^{-} \to 0 \text{ and } n \to \infty.
\end{equation}
Furthermore, extending an optimal policy of this finite-state approximate model, denoted $\hat{\gamma}$, to the entire belief space via the quantization map yields asymptotically optimal performance: the expected cost of applying $\hat{\gamma}$ to the exact discrete-time Euler--Maruyama POMDP satisfies
\begin{equation} \label{eq:policy_convergence}
\sup_{\pi_0 \in \cP(\cK)} \left|J_N(\pi_0, \hat{\gamma}) - J_N^*(\pi_0)\right| \to 0 \quad \text{as } L_n^{-} \to 0 \text{ and } n \to \infty.
\end{equation}

In summary, the continuous-time diffusion with partial information is discretized in time via the Euler--Maruyama scheme. This step ensures the POMDP conditional independence structure and the weak Feller regularity of the filter. This permits the quantization of the belief space. The resulting fully finite, computable MDP generates a policy that is asymptotically near optimal for the intermediate discrete-time model and, by Theorem~\ref{thm:main_approximation}, for the original continuous-time diffusion.

\begin{remark}[Discounted cost criterion]\label{rem:discounted}
While the analysis in this section was presented for the finite-horizon criterion, it is directly applicable to the infinite-horizon discounted cost criterion as well. Recall from Theorem~\ref{thm:main_approximation_discounted} that the discrete-time Euler--Maruyama POMDP under the discounted cost is asymptotically near optimal for the continuous-time problem. By \cite[Theorem 4.29]{SaLiYuSpringer} the quantization of the belief space is directly applicable to the discounted cost problem: constructing the finite-state approximate model via the quantization diameter $L_n^{-}$ yields the uniform convergence
\[
\sup_{\pi_0 \in \cP(\cK)} \left|\hat{J}_\beta(\pi_0) - J_\beta^*(\pi_0)\right| \to 0 \quad \text{as } L_n^{-} \to 0,
\]
for any compact subset $\cK \subset \X$. Similarly to the discussion above, the policy designed for the approximate model is near optimal for the original model. Thus, the entire chain of time, action, and belief discretizations provides a fully computable, near optimal policy for the continuous-time diffusion under the discounted cost criterion as well.
\end{remark}

\subsection{Observation Quantization and the Finite POMDP Approximation}\label{sec:obs_quant}

Having established the quantization of the action and belief spaces, we have derived a model with finite state and action sets. In this section, we demonstrate that the continuous observation space can also be quantized. (Alternatively, when filter stability holds, finite-window policies are near optimal \cite{kara2021convergence}, so that a near optimal policy can be computed under quantized finite measurements; moreover, for the discounted cost criterion, Q-learning algorithm is applicable \cite{kara2021convergence}.)

We first formalize the theoretical framework for observation quantization. We need to assume that the observation kernel admits a density with respect to a reference measure $\lambda(dy)$ (e.g., the Lebesgue measure), such that
\[
Q(dy \mid x) = q(x, y)\, \lambda(dy),
\]
where $q(x, y)$ is the corresponding density function.

Because the observation space $\Y = \R^M$ is not compact, we restrict attention to a large compact ball $K_{\Y} \subset \R^M$ and partition it into disjoint Borel sets $\{B_i\}_{i=1}^{n-1}$; the remainder forms the overflow bin $B_n = \R^M \setminus K_{\Y}$. (Since the observation kernels \eqref{eq:Q_h_def} are Gaussian with uniformly bounded means, the mass of the overflow bin is uniformly small by Gaussian tail bounds once $K_\Y$ is chosen large enough. As in the previous subsection, we may restrict the initial distributions to those with finite second moments; since the cost function is bounded, the cost contribution of the overflow bin vanishes for both the finite-horizon and the discounted criteria.) A finite representative set is defined as
\[
\Y_n := \{ y_1, \dots, y_n \}, \qquad y_i \in B_i,
\]
and the quantization map $\psi_{\Y} : \Y \to \Y_n$ is given by $\psi_{\Y}(y) = y_i$ if $y \in B_i$. Using this quantization map, we define a new observation channel $Q_n$ by
\[
Q_n(y_i \mid x) := Q(B_i \mid x).
\]

\begin{assumption}\label{assum:quantMeasurementL}
The observation density $q(x, y)$ is Lipschitz continuous in $y$, uniformly over $x \in \X$, with Lipschitz constant $\alpha_{\Y}$:
\[
|q(x, y) - q(x, y')| \leq \alpha_{\Y} |y - y'|, \qquad \forall y, y' \in \Y, \; x \in \X.
\]
\end{assumption}

To facilitate the analysis, we define an intermediate observation channel $\tilde{Q}_n$ with density $\tilde{q}$ with respect to $\lambda$, where for $y' \in B_i$,
\[
\tilde{q}(x, y') = \frac{Q(B_i \mid x)}{\lambda(B_i)} = \frac{1}{\lambda(B_i)} \int_{B_i} q(x, y)\, \lambda(dy).
\]
Although $\tilde{Q}_n$ and $Q_n$ operate over different spaces, they are informationally equivalent and thus yield identical value functions \cite{devranby2025near}. The approximation error is governed by the maximum diameter of the bounded quantization bins:
\[
L_{\Y_n} := \max_{i \in \{1, \dots, n-1\}} \sup_{y, y' \in B_i} |y - y'|.
\]
Under Assumption \ref{assum:quantMeasurementL}, \cite{devranby2025near} establishes that the total variation distance between the original and quantized channels satisfies
\[
d_{TV}(\tilde{Q}_n, Q) := \sup_x \|\tilde{Q}_n(\cdot \mid x) - Q(\cdot \mid x)\|_{TV} \leq \alpha_{\Y}\, L_{\Y_n}.
\]

In the finite-horizon setting, this total variation bound directly controls the performance degradation: \cite[Theorem 6.2]{YukselOptimizationofChannels} establishes that for any policy $\gamma$ and observation kernels $Q, Q'$, the cost difference satisfies
\[
\left|J_N(\mu, Q', \gamma) - J_N(\mu, Q, \gamma)\right| \leq \sum_{k=1}^N  \|c\|_\infty \cdot k \cdot \sup_{x \in \X} \|Q'(\cdot \mid x) - Q(\cdot \mid x)\|_{TV}.
\]
Substituting the total variation bound yields a performance guarantee for the observation quantization:
\begin{equation}\label{eq:obs_quant_finiteN}
\left|J_N(\mu, Q_n, \gamma) - J_N(\mu, Q, \gamma)\right| \leq \frac{N (N+1)}{2} \|c\|_\infty\, \alpha_{\Y}\, L_{\Y_n}.
\end{equation}

\begin{remark}[Discounted cost criterion]
For the infinite-horizon discounted cost setting with discount factor $\beta \in (0,1)$, similar bounds are available: \cite[Theorem 6.2]{YukselOptimizationofChannels} and \cite[Theorem 5.4]{devranby2025near} show that using the optimal policy $\hat{\gamma}$ designed for the quantized model results in a bounded performance degradation:
\[
J_\beta(\mu, Q, \hat{\gamma}) - J_\beta^*(\mu, Q) \leq \frac{\beta}{(1 - \beta)^2} \|c\|_\infty\, \alpha_{\Y}\, L_{\Y_n}.
\]
\end{remark}

As the resolution $L_{\Y_n} \to 0$, the performance loss vanishes, proving that the discrete-time model with a finite observation set yields an asymptotically near optimal solution. In the following lemma, we show that the discrete-time Euler--Maruyama observation channel satisfies the Lipschitz density requirement of Assumption \ref{assum:quantMeasurementL}, with an explicit constant.

\begin{lemma}\label{lem:obs_density_lipschitz}
Let the discrete-time observation model be given by the Euler--Maruyama increment \eqref{eq:obs_increment}:
\[
\bar{Y}_{k+1}^h =  g(X^{EM}_{k}) h + (B_{(k+1)h} - B_{kh}),
\]
where $(B_t)_{t\geq 0}$ is an $\R^M$-valued standard Wiener process. Then the conditional distribution of $\bar{Y}_{k+1}^h$ given $X^{EM}_{k} = x$ admits a density $q_h(x, y)$ with respect to the Lebesgue measure on $\R^M$ that is globally Lipschitz continuous in $y$, uniformly in $x$:
\[
\left| q_h(x,y) - q_h(x,y') \right| \leq \alpha_\Y | y - y' |, \qquad \alpha_\Y := \frac{1}{(2 \pi h)^{M/2} \sqrt{e h}}.
\]
\end{lemma}

\begin{proof}
Fix $x \in \R^d$. Since $B_{(k+1)h} - B_{kh} \sim \mathcal{N}(0, h I_M)$ is independent of $X^{EM}_k$, the conditional distribution of the measurement is Gaussian, $\bar{Y}_{k+1}^h \mid \{X^{EM}_k = x\} \sim \mathcal{N}(g(x)h,\, h I_M)$, with density
\[
q_h(x, y) = \frac{1}{(2\pi h)^{M/2}} \exp\left( -\frac{|y - g(x)h|^2}{2h} \right), \qquad
| \nabla_y q_h(x, y) | = \frac{| y - g(x)h |}{h}\, q_h(x, y).
\]

To bound $|q_h(x, y) - q_h(x, y')|$, we apply the mean value theorem:
$$
|q_h(x, y) - q_h(x, y')| \leq \sup_{z \in [y, y']} | \nabla_y q_h(x, z)| \cdot | y - y' |.
$$

Define $r := | z - g(x)h | \ge 0$. The norm of the gradient can be expressed as a function of $r$:
\begin{equation}
v(r) := | \nabla_y q_h(x, z)| = \frac{1}{h (2\pi h)^{M/2}}\, r \exp\left( -\frac{r^2}{2h} \right).
\end{equation}
Differentiating $v(r)$ with respect to $r$ yields:
\begin{equation}
\frac{d}{dr} v(r) = \frac{1}{h (2\pi h)^{M/2}} \exp\left( -\frac{r^2}{2h} \right) \left( 1 - \frac{r^2}{h} \right).
\end{equation}
$\frac{d}{dr} v(r) =0 $ reveals a unique critical point at $r = \sqrt{h}$. Because the derivative is positive for $r < \sqrt{h}$ and negative for $r > \sqrt{h}$, this point is the strict global maximum. Substituting $r = \sqrt{h}$ provides the supremum:
\[
\sup_{z} | \nabla_y q_h(x, z)| = v(\sqrt{h}) = \frac{1}{h (2\pi h)^{M/2}}\, \sqrt{h}\, e^{-1/2} = \frac{1}{(2 \pi h)^{M/2} \sqrt{e h}}.
\]

Therefore, for all $y, y' \in \mathbb{R}^M$, we obtain the global bound:
$$
|q_h(x, y)  - q_h(x, y') | \leq \frac{1}{(2 \pi h)^{M/2} \sqrt{e h}} | y - y' |.
$$
\end{proof}

Consequently, the observation space of the Euler--Maruyama approximation can be discretized via quantization, ensuring that the resulting finite-observation POMDP yields a near optimal policy for the original continuous-time system.

\begin{corollary}\label{cor:finite_obs}
Suppose Assumptions \ref{cont_bdd}, \ref{nondeg}, and \ref{assump:lipschitz} hold, let $c$ and $c_T$ be bounded and Lipschitz continuous, let $g$ be bounded. Given a time step $h$ and quantization parameters, let $\hat{\gamma}_{f}$ denote an optimal policy of the finite-observation POMDP obtained from the Euler--Maruyama approximation through action and observation quantization. Then $\hat{\gamma}_{f}$, applied to the true continuous-time system, is asymptotically near optimal: for every $\varepsilon > 0$ there exist $h > 0$ and quantization parameters such that
\begin{equation}\label{eq:final_eps}
    J_T(x, \hat{\gamma}_{f}) \le J_T^*(x) + \varepsilon,
\end{equation}
where $J_T^*(x)$ is the optimal finite-horizon cost of the continuous-time partially observed diffusion \eqref{cont_model}--\eqref{obs_model}.
\end{corollary}

\section{Conclusion}\label{conclusion}

This paper established a rigorous approximation framework for continuous-time controlled diffusions under partial information. By constructing a discrete-time Euler--Maruyama approximation, we proved that the resulting partially observable model has several regularity properties. Specifically, we established the total variation continuity of the observation kernel and the weak continuity of the state transition kernel, which together guarantee the weak Feller property of the corresponding non-linear filter.

These results justify the quantization of the action, belief, and observation spaces. Through a step-by-step error analysis, we demonstrated that this sequential chain of approximations preserves near optimality: the approximation errors introduced by time discretization and quantization vanish asymptotically, and the Euler--Maruyama substitution error admits an explicit $O(\sqrt{h})$ bound, uniformly over policies. Consequently, an optimal policy derived from the resulting fully finite, computable POMDP is asymptotically near optimal for the original continuous-time system. This establishes a complete, theoretically verified path from intractable continuous-time stochastic control problems to implementable numerical methods.

\bibliographystyle{plain}
\bibliography{SerdarBibliography,DKY,Somnath_Robustness}% common bib file

\begin{thebibliography}{10}

\bibitem{alt2020pomdps}
B.~Alt, M.~Schultheis, and H.~Koeppl.
\newblock {P}{O}{M}{D}{P}s in continuous time and discrete spaces.
\newblock {\em Advances in Neural Information Processing Systems},
  33:13151--13162, 2020.

\bibitem{bain2009fundamentals}
A.~Bain and D.~Crisan.
\newblock {\em Fundamentals of stochastic filtering}, volume~3.
\newblock Springer, 2009.

\bibitem{Jakobsen2006}
Guy Barles and Espen~R. Jakobsen.
\newblock Error bounds for monotone approximation schemes for
  hamilton--jacobi--bellman equations.
\newblock {\em SIAM Journal on Numerical Analysis}, 43(2):540--558, 2005.

\bibitem{benevs1971existence}
V.~E. Bene{\v{s}}.
\newblock Existence of optimal stochastic control laws.
\newblock {\em SIAM Journal on Control}, 9(3):446--472, 1971.

\bibitem{Ber75}
D.P. Bertsekas.
\newblock Convergence of discretization procedures in dynamic programming.
\newblock {\em IEEE Trans. Autom. Control}, 20(3):415--419, Jun. 1975.

\bibitem{bismut1982partially}
J.-M. Bismut.
\newblock Partially observed diffusions and their control.
\newblock {\em SIAM Journal on Control and Optimization}, 20(2):302--309, 1982.

\bibitem{Blackwell2}
D.~Blackwell.
\newblock Memoryless strategies in finite-stage dynamic programming.
\newblock {\em Annals of Mathematical Statistics}, 35:863--865, 1964.

\bibitem{Bog07}
V.I. Bogachev.
\newblock {\em Measure Theory: Volume II}.
\newblock Springer, 2007.

\bibitem{BoArGh12}
V.~S. Borkar, A.~Arapastathis, and M.~K. Ghosh.
\newblock {\em Ergodic Control of Diffusion Processes}.
\newblock Cambridge, UK, Cambridge Uni. Press, 2012.

\bibitem{chigansky2009intrinsic}
P.~Chigansky, R.~Liptser, and R.~van Handel.
\newblock Intrinsic methods in filter stability.
\newblock {\em Handbook of Nonlinear Filtering}, 2009.

\bibitem{cohen2023exponential}
S.N. Cohen and E.~Fausti.
\newblock Exponential contractions and robustness for approximate wonham
  filters.
\newblock {\em arXiv preprint arXiv:2305.02256}, 2023.

\bibitem{cohen2025optimal}
S.N. Cohen, C.~Knochenhauer, and A.~Merkel.
\newblock Optimal adaptive control with separable drift uncertainty.
\newblock {\em SIAM Journal on Control and Optimization}, 63(2):1348--1373,
  2025.

\bibitem{crisan2022}
Dan Crisan, Alexander Lobbe, and Salvador Ortiz-Latorre.
\newblock {\em Pathwise Approximations for the Solution of the Non-Linear
  Filtering Problem}.
\newblock Springer International Publishing, Cham, 2022.

\bibitem{davis1972information}
M.~H.~A Davis and P.~Varaiya.
\newblock Information states for linear stochastic systems.
\newblock {\em Journal of Mathematical Analysis and Applications},
  37(2):384--402, 1972.

\bibitem{davis1973dynamic}
M.~H.~A Davis and P.~Varaiya.
\newblock Dynamic programming conditions for partially observable stochastic
  systems.
\newblock {\em SIAM Journal on Control}, 11(2):226--261, 1973.

\bibitem{diMasi81}
Giovanni~B. Di~Masi and Wolfgang~J. Runggaldier.
\newblock Continuous-time approximations for the nonlinear filtering problem.
\newblock {\em Applied Mathematics and Optimization}, 7(3):233--245, 1981.

\bibitem{diMasi82}
Giovanni~B. Di~Masi and Wolfgang~J. Runggaldier.
\newblock On approximation methods for nonlinear filtering.
\newblock In Sanjoy~K. Mitter and Antonio Moro, editors, {\em Nonlinear
  Filtering and Stochastic Control}, volume 972 of {\em Lecture Notes in
  Mathematics}, pages 249--259. Springer, Berlin, Heidelberg, 1982.

\bibitem{dugundji1951}
J.~Dugundji.
\newblock An extension of {T}ietze's theorem.
\newblock {\em Pacific Journal of Mathematics}, 1:353--367, 1951.

\bibitem{FeKaZg14}
E.A. Feinberg, P.O. Kasyanov, and M.Z. Zgurovsky.
\newblock Partially observable total-cost {M}arkov decision process with weakly
  continuous transition probabilities.
\newblock {\em Mathematics of Operations Research}, 41(2):656--681, 2016.

\bibitem{FlPa82}
W.H. Fleming and E.~Pardoux.
\newblock Optimal control for partially observed diffusions.
\newblock {\em SIAM J. Control Optim.}, 20(2):261--285, 1982.

\bibitem{HernandezLermaMCP}
O.~Hern{\'a}ndez-Lerma and J.~B. Lasserre.
\newblock {\em Discrete-Time {M}arkov Control Processes: Basic Optimality
  Criteria}.
\newblock Springer, 1996.

\bibitem{jakobsen2019improved}
E.~R. Jakobsen, A.~Picarelli, and C.~Reisinger.
\newblock Improved order 1/4 convergence for piecewise constant policy
  approximation of stochastic control problems.
\newblock 2019.

\bibitem{devranby2025near}
A.D. Kara, E.~Bayraktar, and S.~Y\"uksel.
\newblock Near optimal approximations and finite memory policies for pompds
  with continuous spaces.
\newblock {\em Journal of Systems Science and Complexity}, page~0, 2025.

\bibitem{KSYWeakFellerSysCont}
A.D Kara, N.~Saldi, and S.~Y\"uksel.
\newblock Weak {F}eller property of non-linear filters.
\newblock {\em Systems \& Control Letters}, 134:104--512, 2019.

\bibitem{kara2021convergence}
A.D Kara and S.~Y\"uksel.
\newblock Convergence of finite memory {Q}-learning for {POMDP}s and near
  optimality of learned policies under filter stability.
\newblock {\em Mathematics of Operations Research}, 48(4):2066–2093, 2023.

\bibitem{KaSh91}
I.~Karatzas and S.~E. Shreve.
\newblock {\em Brownian Motion and Stochastic Calculus}, volume 113 of {\em
  Graduate Texts in Mathematics}.
\newblock Springer, 2 edition, 1991.

\bibitem{krylov1999approximating}
N.V. Krylov.
\newblock Approximating value functions for controlled degenerate diffusion
  processes by using piece-wise constant policies.
\newblock 1999.

\bibitem{krylov2000rate}
N.V. Krylov.
\newblock On the rate of convergence of finite-difference approximations for
  {B}ellman's equations with variable coefficients.
\newblock {\em Probability Theory \& Related Fields}, 117(1), 2000.

\bibitem{krylov2008controlled}
N.V. Krylov.
\newblock {\em Controlled Diffusion Processes}, volume~14.
\newblock Springer, 2008.

\bibitem{Kullback59}
S.~Kullback.
\newblock {\em {Information Theory and Statistics}}.
\newblock Wiley, New York, 1959.

\bibitem{kushner1990numerical}
H.~J. Kushner.
\newblock Numerical methods for stochastic control problems in continuous time.
\newblock {\em SIAM Journal on Control and Optimization}, 28(5):999--1048,
  1990.

\bibitem{kushner2014partial}
H.~J. Kushner.
\newblock A partial history of the early development of continuous-time
  nonlinear stochastic systems theory.
\newblock {\em Automatica}, 50(2):303--334, 2014.

\bibitem{kushner2001numerical}
H.~J. Kushner and P.~G. Dupuis.
\newblock {\em Numerical Methods for Stochastic Control Problems in Continuous
  Time}, volume~24.
\newblock Springer Science \& Business Media, 2001.

\bibitem{kushner2012weak}
H.J. Kushner.
\newblock {\em Weak convergence methods and singularly perturbed stochastic
  control and filtering problems}.
\newblock Springer Science \& Business Media, 2012.

\bibitem{martinez2009}
Miguel Martinez, Sylvain Rubenthaler, and Etienne Tanr{\'e}.
\newblock Approximations of a continuous time filter: Application to optimal
  allocation problems in finance.
\newblock {\em Stochastic Analysis and Applications}, 27(2):270--296, 2009.

\bibitem{Munkres}
J.~R. Munkres.
\newblock {\em Topology}.
\newblock Prentice Hall, 2 edition, 2000.

\bibitem{PY25discapprox}
S.~Pradhan and S.~Y\"{u}ksel.
\newblock Discrete-time approximations of controlled diffusions with infinite
  horizon discounted and average cost.
\newblock {\em arXiv2502.05596}, 2025.

\bibitem{pradhanyuksel2023DTApprx}
Somnath Pradhan and Serdar Y{\"u}ksel.
\newblock Controlled diffusions under full, partial, and decentralized
  information: Existence of optimal policies and discrete-time approximations.
\newblock {\em SIAM Journal on Control and Optimization}, 63(5):3674--3702,
  2025.

\bibitem{Rhe74}
D.~Rhenius.
\newblock Incomplete information in {M}arkovian decision models.
\newblock {\em Ann. Statist.}, 2:1327--1334, 1974.

\bibitem{SaLiYuSpringer}
N.~Saldi, T.~Linder, and S.~Y\"uksel.
\newblock {\em Finite Approximations in Discrete-Time Stochastic Control:
  Quantized Models and Asymptotic Optimality}.
\newblock Springer, Cham, 2018.

\bibitem{saldiyukselGeoInfoStructure}
N.~Saldi and S.~Y\"uksel.
\newblock Geometry of information structures, strategic measures and associated
  control topologies.
\newblock {\em Probability Surveys}, 19:450--532, 2022.

\bibitem{SaYuLi15c}
N.~Saldi, S.~Y{\"u}ksel, and T.~Linder.
\newblock On the asymptotic optimality of finite approximations to {M}arkov
  decision processes with {B}orel spaces.
\newblock {\em Mathematics of Operations Research}, 42(4):945--978, 2017.

\bibitem{SYLTAC2017POMDP}
N.~Saldi, S.~Y\"uksel, and T.~Linder.
\newblock Finite model approximations for partially observed {M}arkov decision
  processes with discounted cost.
\newblock {\em IEEE Transactions on Automatic Control}, 65, 2020.

\bibitem{tallec2019making}
C.~Tallec, L.~Blier, and Y.~Ollivier.
\newblock Making deep {Q}-learning methods robust to time discretization.
\newblock In {\em International Conference on Machine Learning}, pages
  6096--6104. PMLR, 2019.

\bibitem{tsybakov2009}
A.~B. Tsybakov.
\newblock {\em Introduction to Nonparametric Estimation}.
\newblock Springer, New York, 2009.

\bibitem{YukselOptimizationofChannels}
S.~Y\"uksel and T.~Linder.
\newblock Optimization and convergence of observation channels in stochastic
  control.
\newblock {\em SIAM J. on Control and Optimization}, 50:864--887, 2012.

\bibitem{Yus76}
A.A. Yushkevich.
\newblock Reduction of a controlled {M}arkov model with incomplete data to a
  problem with complete information in the case of {B}orel state and control
  spaces.
\newblock {\em Theory Prob. Appl.}, 21:153--158, 1976.

\end{thebibliography}
\end{document}